\documentclass[12pt,a4paper,twoside]{amsart}
\usepackage{a4wide,times, amsmath,amsbsy,amsfonts,amssymb,
stmaryrd,amsthm,mathrsfs,graphicx, amscd, tikz-cd, cancel}
\usepackage{dsfont}
\usepackage[all]{xy}
\usepackage{xypic}
\usepackage{yhmath}
\usepackage{mathrsfs}
\usepackage{url}

\usepackage{fancyhdr}
\usepackage{lastpage}  
\usepackage{xcolor}
\usepackage{listings}
\usepackage[normalem]{ulem}
\usetikzlibrary{matrix,arrows,decorations.pathmorphing}
\usepackage[active]{srcltx}
\usepackage[
	hypertexnames=false,
	hyperindex,
	pagebackref,
	pdftex,
	breaklinks=true,
	bookmarks=true,
	colorlinks,
	linkcolor=blue,
	citecolor=red,
	urlcolor=red,
]{hyperref}
\usepackage[mathscr]{eucal}

\newcommand\GL{\operatorname{GL}}
\newcommand\SL{\operatorname{SL}}

\newcommand\SO{\operatorname{SO}}

\newcommand\Sp{\operatorname{Sp}}

\newcommand\edge{\textrm{edge}}
\newcommand\trc{\textrm{trc}}

\newcommand\hyp{\textrm{hyp}}

\newcommand\cel{\textrm{cell}}
\newcommand\ini{\textrm{in}}
\newcommand\tm{\textrm{tm}}

\newcommand\bound{\textrm{bound}}
\newcommand\la{{\langle}}
\newcommand\ra{{\rangle}}

\newcommand\Aut{\operatorname{Aut}}

\newcommand\End{\operatorname{End}}

\newcommand\Ext{\operatorname{Ext}}
\newcommand\Hom{\operatorname{Hom}}

\newcommand\Img{\operatorname{im}}
\newcommand\Id{\operatorname{Id}}

\newcommand\Stab{\operatorname{Stab}}
\newcommand\coker{\operatorname{coker}}

\newcommand\Acal{\mathcal A}

\newcommand\Kcal{\mathcal K}
\newcommand\Ocal{\mathcal O}
\newcommand\Scal{\mathcal S}

\newcommand\Csr{\mathscr C}
\newcommand\Bsr{\mathscr B}

\newcommand\Isr{\mathscr I}

\newcommand\Lsr{\mathscr L}

\newcommand\Rsr{\mathscr R}

\newcommand\Wsr{\mathscr W}

\newcommand\Cds{\mathds{C}}

\newcommand\Fds{\mathds{F}}
\newcommand\Hds{\mathds{H}}

\newcommand\Pds{\mathds{P}}
\newcommand\Qds{\mathds{Q}}
\newcommand\Rds{\mathds{R}}

\newcommand\Zds{\mathds{Z}}

\makeatletter

\newcommand{\Rnum}[1]{\expandafter\@slowromancap\romannumeral #1@}

\newtheorem{theorem}{Theorem}[section]
\newtheorem{lem}[theorem]{Lemma}  
\newtheorem{thm}[theorem]{Theorem}  
\newtheorem{cor}[theorem]{Corollary}  
\newtheorem{prop}[theorem]{Proposition}

\theoremstyle{definition}
\newtheorem{rmk}[theorem]{Remark}

\begin{document}
\title{On the Monodromy and Period Map of the Winger Pencil}  
\author{Yunpeng Zi}  
\date{}  
\maketitle

\begin{abstract}
    The sextic plane curves that are invariant under the standard action of the icosahedral group on the projective plane make up a pencil of genus ten curves  (spanned by a sum of six lines and a three times a  conic). This pencil was first considered in a note by R.~M.~Winger in 1925 and is nowadays named after him.
    We gave this a modern treatment and proved among other things that it contains essentially every smooth genus  ten  curve with  icosahedral  symmetry.
    We here consider the monodromy group and the period map naturally defined by the icosahedral symmetry. We showed that this monodromy group is a subgroup of finite index in $\SL_2(\Zds[\sqrt{5}])$ and the period map brings the Winger pencil to a curve on the Hilbert modular surface $\SL_2(\Zds[\sqrt{5}])/\Hds^2$.
\end{abstract}

\section{Introduction}
This is the last part of a series of paper concerning the Winger's pencil and also a continuation of the author's Phd thesis. The Winger's pencil $\Wsr_{\Bsr}$ is a linear system of planer genus 10 curves with $\Acal_5$-symmetry on each fiber which is studied in \cite{wingerinvariants} by R.M\ Winger.
If $I$ is a complex 3-space endowed with a faithful $\Acal_5$-action, it is defined as a hypersurface by the following equation in the projective variety $\Pds(I)\times \Bsr\cong \Pds^2\times\Pds^1$
\begin{equation}
    g_2^3+tg_6=0
\end{equation}
Here $t\in \Bsr$ be a parameter, $g_2$ and $g_6$ are two generators of $\Cds[I]^{\Acal}_6$ where $g_2$ is a polynomial of degree two representing a smooth conic and $g_6$ is a polynomial of degree 6 representing the union of 6 lines. The singular members of this pencil appears only at four points, they are $t=-1$ an irreducible curve with six nodes which is also coming from identifying six pairs of double points on the Bring's curve, $t=0$ a smooth conic with multiple three, $t=\frac{27}{5}$ an irreducible curve with ten node and the last $t=\infty$ the union of six lines.

We have showed in \cite{zi2021geometry} that with some modifications the new object "Winger's family" parameterized all stable genus 10 curves with $\Acal_5$-symmetry.
It was showed in the same paper that for a smooth member $C_t$ ($t\in \Bsr^{\circ}$ a point in the smooth locus of $\Wsr\to\Bsr$) of the Winger pencil, its space of holomorphic forms $H^0(C,\omega_{C})$ is isomorphic to $V\oplus I\oplus I'=V\oplus E$ as a $\Cds \Isr$-module where $V$ is the permutation representation of dimension four and $I$ and $I'$ are three dimensional irreducible representations. This implies that $H^1(C;\Cds)$ is isomorphic to $V^{\oplus2}\oplus E^{\oplus 2}$. Since both $V$ and $E$ are complexifications of irreducible $\Qds\Isr$-modules $V_\Qds$ resp. $E_\Qds$ (which are therefore self-dual), it follows that there exist a canonical isotypical decomposition for  $H_1(C;\Qds)$
\begin{equation}\label{eqn:homologydec}
    H_1(C_t;\Qds)\cong(V_\Qds\otimes \Hom_{\Qds\Isr}(V_\Qds,H_1(C_t,\Qds)))\oplus(E_\Qds\otimes \Hom_{\Qds\Isr}(E_\Qds,H_1(C_t,\Qds)))
\end{equation}
with $\dim_\Qds \Hom_{\Qds\Isr}(V_\Qds,H_1(C_t,\Qds))=2$ and $\dim_{\Qds(\sqrt{5})} \Hom_{\Qds\Isr}(E_\Qds,H_1(C_t, \Qds))=2$.
We have proved in \cite{looijenga2021monodromy} that the (global) monodromy restricted to $V$-part has image in $\SL_2(\Zds)$, more explicitly it is the index $8$ congruence subgroup $\Gamma_1(3)$ of $\SL_2(\Zds)$. Hence the period map $p_{V}:\Bsr^{\circ}\to \Hds/\SL_2(\Zds)$ is a ramified finite morphism of degree 8.

In this paper, we will focus on the $E$-part of the decomposition \eqref{eqn:homologydec}. If $E_o$ is a fixed integral form of the representation $E$ with endomorphism ring $\Ocal_o$, the monodromy group and period map related to $E_o$ is denoted by $\Gamma_{E_o}$ and $p_{E_o}$. We could also observe that there exist an inner product on $E_o$ and a symplectic form on $H_1(C_t)$, the monodromy action will preserve these forms. This implies that the monodromy group must be a subgroup of $\Sp_1(\Ocal_o)\cong\SL_2(\Ocal_o)$. The main theorems are the following:
\begin{thm}\label{thm:main1}
    The monodromy group $\Gamma_{E_o}$ is a subgroup of finite index in $\SL_{2}(\Ocal_o)$. In particular it is arithmetic.
\end{thm}
And if $\Bsr^{+}$ be the open subvariety of $\Bsr$ obtained by removing from $\Bsr$ the three points representing nodal curves, we have the following theorem about the period map.
\begin{thm}
    The 'partial' period map $p_{E_o}:\Bsr^{+}\to\Gamma_{E_o}/\Hds^2 \to\SL_2(\Ocal_o)/\Hds^2$ has the property that the first arrow is open and the second map is finite.
\end{thm}
Moreover with the help of computer program \textit{Magma}, we could say a little more property about the group $\Gamma_{E_o}$ namely
\begin{thm}\label{thm:mainind}
    The monodromy group $\Gamma_{E_o}$ is of index two in $\SL_{2}(\Ocal_o)$.
\end{thm}

The main tool of surveying the the monodromies and the periods are two models of the genus 10 curve with an $\Acal_5$-symmetry.
The first one which we named it as $\Sigma$ is coming from the regular icosahedron with a natural $\Acal_5$-action by removing in a $\Acal_5$-equivalent manner a small triangle at each vertices and identifying the antipodal points on the boundary. This is also the model that we used in \cite{looijenga2021monodromy}.
The second which we call it $\Pi$ is modified from the Euclidean realization of the Bring's curve. This realization is a regular polygon endowed with a $\Acal_5$-symmetry namely the \textit{Great Dodecahedron}. We will remove in a $\Acal_5$-equivalent manner a small pentagram at each vertices and identifying the antipodal points on the boundary. Each of the models give a real one-dimensional family $\Sigma_t$ resp. $\Pi_t$ on the Winger pencil such that they connect two different singular members of the Winger pencil. Then instead of computing the local monodromies on the Winger pencil, we could done it on the family $\Sigma_t$ or $\Pi_t$.

This paper is organized as following, we will introduce some basic lemmas and fix notations after the introduction. And we will take the next two sections devoting to introduce the details of the two models. We will use all these information to determine the local monodromies in Section 4. A global description to the monodromy group $\Gamma_{E_o}$ and period map $p_{E_o}$ will be given in the Section 5. And in the last section we will give the way of computing the index in the last section. As we have talked above, this computation is made by the computer program, the code for this computation has uploaded to \cite{indexprog}.
\subsection{Acknowledgement}
The author wants to thank Prof. Eduard Looijenga for his kind help and useful discussion.

\subsection{The Integral form of $E_{\Qds}$}\label{seclatt}
Before we study this project in detail, let us introduce some properties of $E$ the $6$-dimensional linear representation of $\Acal_5$. Let $I_{\Rds}$ be the Euclidean vector space with a faithful $\Acal_5$-action and we will denote the image of $\Acal_5\hookleftarrow \GL(I_{\Rds})$ as $\Isr$. The $\Rds\Isr$-module $I_{\Rds}$ is irreducible, even its complexification $I$ is an irreducible $\Cds\Isr$-module, but $\Isr$ is not definable over $\Qds$. If $I'$ is obtained from $I$ by precomposing the $\Isr$-action with an outer automorphism of $\Isr$, then $E:=I\oplus I'$ is as a representation naturally defined over $\Qds$. And the character computation shows that it actually splits over the field $\Qds[\sqrt{5}]$.

Let us take $V_o$ be the integral permutation representation of $\Acal_5$ of rank 4 the same as the notations in \cite{looijenga2021monodromy}. Recall that if we take $\Zds^5$ to be the free $\Zds$-module generated by $\{e_i\}_{i=1}^5$ and $\Acal_5$ acts on the set of generators in a natural way, the $\Zds\Acal_5$-module $V_o$ is defined by the following exact sequence.
\begin{equation*}\label{exsqV}
    0\to \Zds\to \Zds^5\to V_o\to 0
\end{equation*}
This exact sequence gives an a surjective map $\wedge^2\Zds^5\to \wedge^2 V_o$ whose kernel is identified with $\Zds^5\wedge (\sum_{i=1}^5e_i)$, so that we have the exact sequence of $\Zds\Acal_5$-modules
\begin{equation}\label{exactE}
    0\to V_o\to \wedge^2 \Zds^5\to \wedge^2 V_o\to 0
\end{equation}
It is clear to see that $\wedge^2 V_o$ is an integral form of $E$ from the character computations. We will always denote $E_{\Qds}$ to be the vector space $\wedge^2 V_o\otimes \Qds$.
Using the similar notion of \cite{farb2021Arithmeticity} we will take $f_{i,j}$ as the image of $e_i\wedge e_j$ in $\wedge^2 V_o$. Let $\phi:\Zds^5\to \Zds$ be the morphism of taking the coordinate sum.
We will denote by $E_o$ the space generated by the elements $(f_{i,j}+f_{j,k}+f_{k,i})$ for all $i,j,k\in\{1,2,3,4,5\}$. Note that $E_o$ is the image of the $\Zds\Isr$-homomorphism
\begin{equation*}
    \delta: \xymatrix{\wedge^3 \Zds^5\ar[r]^{\iota_{\phi}}&\wedge^2\Zds^5 \ar[r]&\wedge^2 V_o}
\end{equation*}
Here $\iota_{\phi}$ is taking inner product with $\phi$. The following Lemma is the Lemma 2.1 of \cite{farb2021Arithmeticity}.
\begin{lem}\label{lem:basisEo}
    Let $e:=\sum_i f_{i,i+1}\in \wedge^2 V_o$. Then the $\Isr$-orbit of $e$ is the union of a basis of $E_o$ and its antipodal. Hence the $\Zds\Isr$-module $E_o$ is principal.
    Moreover there exists an inner product
    \begin{equation*}
        s:E_o\times E_o\to \Zds
    \end{equation*}
    for which this basis is orthogonal is $\Isr$-invariant.
\end{lem}
\begin{rmk}\label{homsym}
    There is a simple observation that if we assume that $(H,\la-,-\ra)$ is a $\Zds\Acal_5$-module endowed with a $\Acal_5$-invariant symplectic form $\la-,-\ra$, the inner product $s$ and the symplectic form $\la-,-\ra$ gives a symplectic form on $\Hom_{\Zds\Isr}(E_o,H)$ in a natural way, since $\Hom_{\Zds\Isr}(E_o,H)$ is a submodule of $\Hom(E_o,H)$. This form is also symplectic form and making $\Hom_{\Zds\Isr}(E_o,H)$ a symplectic $\Zds$-module.
\end{rmk}

Since $E_{\Qds}$ is not absolutely irreducible, there must have endomorphisms which not a multiple of $\Id$. We will construct one such example and show that it is defined over integers and generates the endomorphism ring $\End_{\Zds\Isr}(E_o)$. First let us take the generators of $\Acal_5$ as $\sigma_2:=(1,5)(3,4)$, $\sigma_3:=(2,5,3)$ and $\sigma_5=(1,2,3,4,5)$. It is clear that they satisfies the relation that $\sigma_2\sigma_3\sigma_5=1$. Note that $\sigma_5$ fixes $e$ and $\sigma_2\sigma_5^3\sigma_2\sigma_5^2\sigma_2=(2,5)(3,4)$ maps $e$ to $-e$.
Then the Lemma \ref{lem:basisEo} implies that the basis of $E_o$ is the following
\begin{equation*}
    \{e,e_0:=\sigma_2(e),e_1:=\sigma_5(e_0),\cdots,e_4:=\sigma_5^4(e_0)\}
\end{equation*}
Their relations are as following
\begin{equation*}
    \begin{aligned}
        \sigma_5: & e_0\to e_1\to e_2\to e_3\to e_4\to e_0; \hbox{fixes $e$}                                                              \\
        \sigma_2: & e\leftrightarrow e_0;e_1\leftrightarrow e_4;e_2\leftrightarrow -e_2;e_3\leftrightarrow -e_3; \hbox{fixes $(e+e_0  )$} \\
        \sigma_3: & e\to e_0\to e_1\to e;e_2\to e_4\to -e_3\to e_2; \hbox{fixes $(e+e_0+e_1)$}
    \end{aligned}
\end{equation*}
We will take the endomorphism $X\in \End(E_o)$ as following
\begin{equation*}
    \begin{aligned}
        X(e_{}) & :=e_0+e_1+e_2+e_3+e_4 \\
        X(e_0)  & :=e+e_1-e_2-e_3+e_4   \\
        X(e_1)  & :=e+e_0+e_2-e_3-e_4   \\
        X(e_2)  & :=e-e_0+e_1+e_3-e_4   \\
        X(e_3)  & :=e-e_0-e_1+e_2+e_4   \\
        X(e_4)  & :=e+e_0-e_1-e_2+e_3
    \end{aligned}
\end{equation*}
It is clear to check that $\sigma_i X=X\sigma_i$ for all $i\in\{2,3,5\}$. Hence it is a nontrivial element of $\End_{\Zds\Isr}(E_o)$ which is not a multiple by an integer. Moreover $X$ satisfies the relation that $X^2-5\Id=0$.
\begin{prop}\label{ringOo}
    The endomorphism ring $\Ocal_o:=\End_{\Zds\Isr}(E_o)$ is generated by $X$ subjects to the relation $X^2-5\Id=0$. Hence it is isomorphic to the quadratic algebraic integers $\Zds[\sqrt{5}]$.
\end{prop}
\begin{proof}
    This is the Lemma 2.8 of \cite{farb2021Arithmeticity}.
\end{proof}
From the Proposition \ref{ringOo}, we get the following Corollaries. The first one \ref{fieldKo} is a more explicit description of the endomorphism of $E_{\Qds}$ and the second \ref{autgoEo} concerns about the automorphism ring of $E_o$.
\begin{cor}\label{fieldKo}
    The endomorphism ring $\Kcal:=\End_{\Qds\Isr}(E_{\Qds})$ of $E_{\Qds}=E_o\otimes \Qds$  is generated by $X$ subjects to the relation $X^2-5\Id=0$. Hence it is isomorphic to the quadratic field extension $\Qds[\sqrt{5}]/\Qds$.
\end{cor}
\begin{cor}\label{autgoEo}
    The automorphism group $\Aut_{\Zds\Isr}(E_o)$ is a multiplicative group cyclic of order two. More explicitly it is the group $<\pm \Id>$.
\end{cor}
\begin{proof}\textbf{(Proof of Corollary \ref{autgoEo})}
    Let $Y\in\Aut_{\Zds\Isr}(E_o)$ be an automorphism of $E_o$. Since $Y$ and $Y^{-1}$ are both endomorphisms of $E_o$, there exist integers $a_1$, $b_1$,$a_2$ and $b_1$ such that $Y=a_1X+b_1\Id$ and $Y^{-1}=a_2X+b_2\Id$. From the Lemma \ref{ringOo}, the relation $YY^{-1}=\Id$ implies that $a_1=a_2=0$ and $b_1=b_2=\pm 1$. This finishes the proof.
\end{proof}
Despite $E_o$ there are other integral forms of $E_{\Qds}$. For example we may take $E\subset E_o$ to be the subspace which has even coefficients sum with respect to the basis in \ref{lem:basisEo}. It is a sublattice of index two in $E_o$ and it was proved in the Lemma 2.8 of \cite{farb2021Arithmeticity} that $\Ocal:=\End_{\Zds\Isr}(E)$ is isomorphic to to the ring of algebraic integers in $\Qds[\sqrt{5}]$ i.e. $\Zds[Y]/(Y^2-Y-1)$. Moreover there is an embedding  $\Ocal_o\hookrightarrow \Ocal$ given by $X\to (2Y-1)$.


\subsection{Criterion of Generating a Lattice}
As the end of this section, let us introduce some criterions about when a set of elements become the generators of a lattice.
\begin{lem}\label{genlatt}
    Let $(H,\la-,-\ra)$ be a lattice of rank $n$ with bilinear form $\la-,-\ra$. Assume that $a_1,\cdots,a_n$ be elements of $H$ such that they form a $\Qds$-basis of $H_{\Qds}:=H\otimes \Qds$.  If for every $n$ coprime integers $\{\alpha_1,\cdots,\alpha_n\}$ i.e. $\gcd(\alpha_1,\cdots,\alpha_n)=1$, there exist $y(\alpha_1,\cdots,\alpha_n)\in H$ such that
    \begin{equation*}
        \la\sum_{i=1}^n\alpha_i a_i,y(\alpha_1,\cdots,\alpha_n)\ra=1
    \end{equation*}
    then $a_1,\cdots,a_n$ generates $H$ over $\Zds$. Furthermore if we assume $(H,\la-,-\ra)$ is unimodular the converse also holds.
\end{lem}
\begin{proof}
    (1)\ Let us assume that $a_1,\cdots,a_n$ cannot generates $H$ over $\Zds$ and denote the sublattice generated by $a_1,\cdots,a_n$ by $H'$. Since $\{a_1,\cdots,a_n\}$ is a $\Qds$-basis of $\Qds$-vector space $H_{\Qds}$, $H'$ must have the same rank as $H$.
    There exist an element $x\in H$ but $x\notin H'$ satisfying that for every positive integer $k>1$, $\frac{x}{k}\notin H$ and there exist minimal positive integer $m\in\Zds$ such that $m\neq 1$ and $mx\in H'$.
    Hence there exist integers $\alpha_1,\cdots,\alpha_n$ such that $mx=\sum_{i=1}^n\alpha_i a_i$. By the minimality of $x$ and $m$, $\alpha_1,\cdots,\alpha_n$ has the property that $(\alpha_1,\cdots,\alpha_n)=1$. Hence there exist an element $y(\alpha_1,\cdots,\alpha_n)\in H$, such that $\la\sum_{i=1}^n\alpha_i a_i,y(\alpha_1,\cdots,\alpha_n)\ra=1=m\la x,y(\alpha_1,\cdots,\alpha_n)\ra$, which is not possible.

    (2)\ Let us assume that $a_1,\cdots,a_n$ generates $H$ and $(H,\la-,-\ra)$ is unimodular.
    Let $\alpha_1,\cdots,\alpha_n$ be $n$ integers satisfying $\gcd(\alpha_1,\cdots,\alpha_n)=1$. Hence there exist integers $\beta_i$ such that $\sum_{i=1}^n\alpha_i\beta_i=1$. Let us take $y(\alpha_1,\cdots,\alpha_n)\in H^{\vee}=H$ such that $\la a_i,y(\alpha_1,\cdots,\alpha_n)\ra=\beta_i$. Hence we have $\la\sum_{i=1}^n\alpha_i a_i,y(\alpha_1,\cdots,\alpha_n)\ra=\sum_{i=1}^n\alpha_i\beta_i=1$.
\end{proof}
Using the similar argument, we can prove the following Lemma, which is frequently used in the material below.

\begin{lem}\label{genhom}
    Let $(H,\la-,-\ra)$ be an unimodular lattice of finite rank such that $H$ is a $\Zds\Acal_5$-module and $\la-,-\ra$ is an $\Acal_5$-invariant bilinear form. Assume that $\Hom_{\Zds\Acal_5}(E_o,H)$ is free $\Zds$-module of rank $n$ and $\phi_1,\cdots,\phi_n$ are $n$ linearly inequivalent elements of $\Hom_{\Zds\Acal_5}(E_o,H)$ such that they form a basis of $\Qds$-vector space $\Hom_{\Qds\Acal_5}(E_{\Qds},H_{\Qds})$. The elements $\phi_1,\cdots,\phi_n$ generates $\Hom_{\Zds\Acal_5}(E_o.H)$ over $\Zds$ if and only if for every set of $n$ integers $\{\alpha_1,\cdots,\alpha_n\}$ satisfying $\gcd(\alpha_1,\cdots,\alpha_n)=1$, there exist $y(\alpha_1,\cdots,\alpha_n)\in H$ such that
    \begin{equation*}
        \la\sum_{i=1}^n\alpha_i\phi_i(e),y(\alpha_1,\cdots,\alpha_n)\ra=1
    \end{equation*}
\end{lem}
\begin{proof}
    The proof is similar as above.
\end{proof}

\section{Geometric Model from Icosahedron}\label{Sig}

We introduced a geometric model in Section 2 of \cite{looijenga2021monodromy} for a smooth fiber $C_t$ with $t\in\Bsr^{\circ}$ and describe two stable degenerations in terms of it. Here we will give a quick summary on this model without proof. Note that we will always use $z$ to denote a 2-cell or a face of a polyhedron, $y$ to denote a 1-cell or an edge and $x$ to denote a 0-cell or a vertex beginning from this section. For an oriented edge $y$, let $\ini(y)$ be its initial point and $\tm(y)$ be its terminal point.

Let us fix an oriented euclidean 3-space $I_{\Rds}$ and a regular dodecahedron $\tilde{\Sigma}\subset I_{\Rds}$ centered at the origin. Let $\iota$ be the antipodal map which is orientation reversing. Note that the automorphism group $\Isr\subset \SO(I_{\Rds})$ of $\tilde{\Sigma}$ is isomorphic to $\Acal_5$, the alternating group of five elements.
Let $\hat{\Sigma}$ be obtained from the dodecahedron $\tilde{\Sigma}$ by removing in a $\Isr$-invariant manner a small regular triangle  centered at each vertex of $\tilde{\Sigma}$ so that the faces of $\hat{\Sigma}$ are oriented solid $10$-gons.
We now identify opposite points on the boundary of $\hat\Sigma$ and thus obtain a complex $\Sigma$ that is a closed oriented surface of genus $10$ endowed with an action of $\Isr$ (See Figure \ref{trdo}).
\begin{figure}\label{trdo}
    \centering
    {\includegraphics[width=0.2\textwidth]{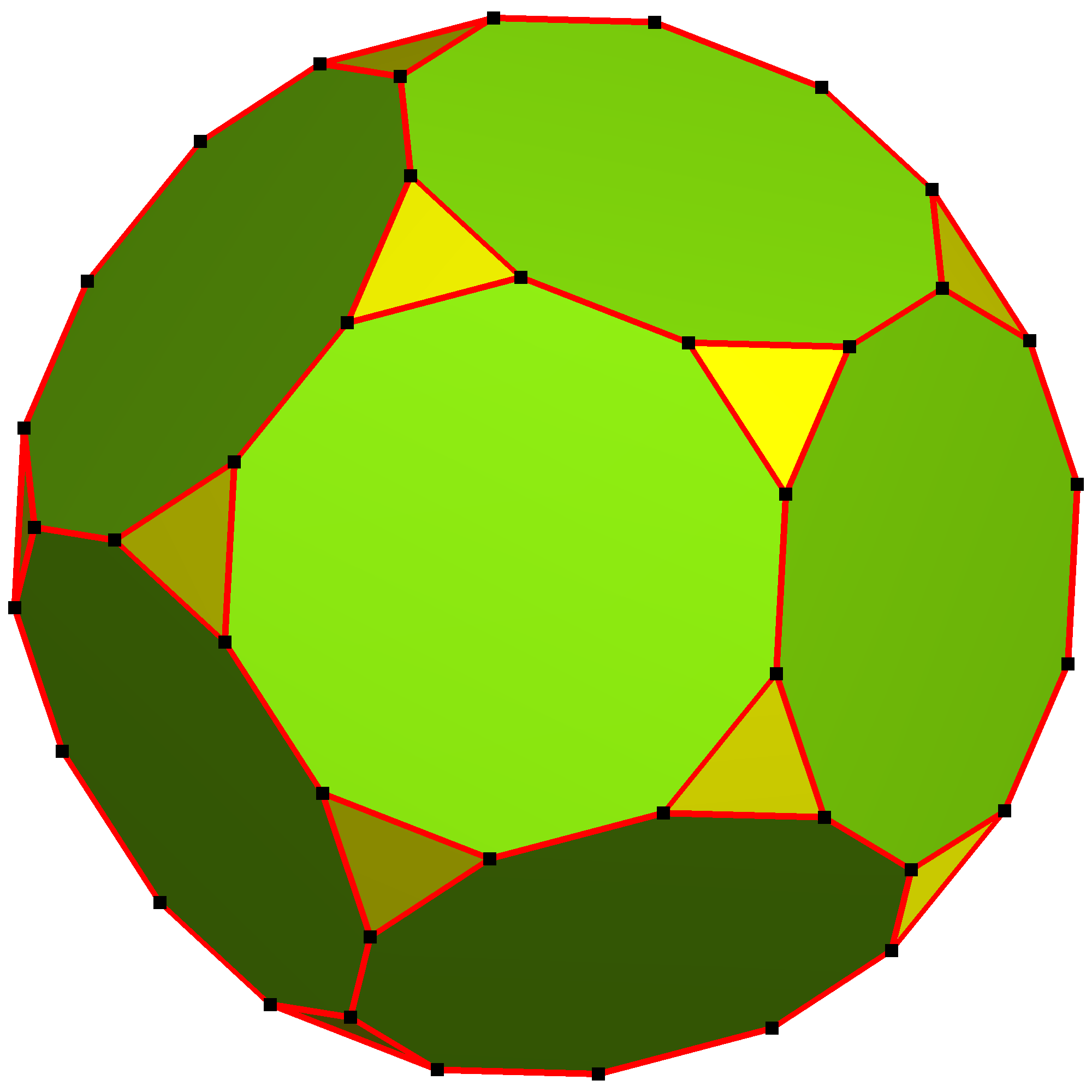}}
    \caption{\small{Removing in a $\Isr$-invariant manner a small regular triangle  centered at each vertex of $\tilde{\Sigma}$}}
    \label{hypemod1}
\end{figure}

Recall that there exist two kinds of oriented 1-cells resp. 1-cycles on $\Sigma$. The ones coming from the truncation is called \emph{$1$-cells resp. 1-cycles of truncation type}. The 1-cycles of truncation type are bijectively indexed by the set $\Csr_0(\tilde{\Sigma})$ of vertices of $\tilde{\Sigma}$, namely for each vertex $x\in\Csr_0(\tilde{\Sigma})$ the sum of three 1-cells of truncation type together with its counterclockwise orientation is a 1-cycle of truncation type. This labeling is denoted by $\delta_x$, the set of all 1-cells of truncation type is denoted by $\Delta_{\trc}$ and we have $\delta_{\iota x}=-\delta_x$.
The ones coming from the edges of $\tilde{\Sigma}$ is called \emph{$1$-cells resp. 1-cycles of edge type}. The 1-cycles of edge type are indexed by $\Csr_1(\tilde{\Sigma})$, namely for each oriented edges $y\in\tilde{\Sigma}$ the division $(y-\iota y)$ is a 1-cycle of edge type. This labeling is denoted by $\delta_y$ and the set of all 1-cells of edge type is denoted by $\Delta_{\edge}$ . Note that this labeling is not bijective for we have the relation that $\delta_{-y}=\delta_{\iota y}=-\delta_y$.

The polyhedron $\Sigma$ can be endowed with a natural complex structure $J_{\tau}$ where $\tau$ is the length of the 1-cells of truncation type such that the following proposition follows
\begin{prop}\label{degsig}(Proposition 3.4 of \cite{looijenga2021monodromy})
    The Riemann surface $(\Sigma, J_\tau)$ is the set of complex points of a complex real algebraic curve. It has genus 10 and comes with a faithful $\Isr$-action, hence is isomorphic to a member of the Winger pencil.
    We thus have defined a continuous map $\gamma: [0,1]\to \Bsr$ which traverses the real interval  $[\infty, \frac{27}{5}]$
    and  which maps $(0,1)$ to $\Bsr^\circ$ (and so lands in the locus where $t$ is real and $>\frac{27}{5}$),
    such that the pull-back of the Winger pencil  yields the family constructed above. The degenerations of $\Sigma$ into $\Sigma_\edge$ resp. $\Sigma_\trc$ have $\Delta_\edge$ resp.\ $\Delta_\trc$ as their sets of vanishing cycles.
\end{prop}

Moreover it is clear to check that the intersection number of the above 1-cycles are as the following Lemma.
\begin{lem}\label{lemma:intsce}
    The intersection numbers of these $1$-cycles are as follows: any two loops of the same type have intersection number zero and if $x\in \Csr_0(\tilde{\Sigma})$ and $y\in \Csr_1(\tilde{\Sigma})$, then $\la \delta_x, \delta_y\ra =0$ unless $x$ lies on $y$ or on $\iota y$, in which case $\la \delta_x, \delta_y\ra\in\{\pm 1\}$ with the plus sign appearing if and only if $x$ is the end point of $y$.
\end{lem}

\subsection{Celluar Homology of $\Sigma$}

Recall that the surface $\Sigma$ comes with a cellular decomposition. Hence there exist a natural exact sequence i.e. the cellular decomposition of  $\Sigma$ enables us compute its homology as that of the combinatorial chain complex
\begin{equation}\label{seqC2}
    \xymatrix{
        0\ar[r]&\Csr_2(\Sigma)\ar[r]^{\partial_2} &\Csr_1(\Sigma)\ar[r]^{\partial_1}& \Csr_0(\Sigma)\ar[r]& 0.
    }
\end{equation}
Let us take $B_i(\Sigma):=\Img{(\partial_{i+1})}$ and $Z_i(\Sigma):=\ker(\partial_i)$. Hence we have the following two exact sequences where the second describes the first homology group. The first describes the 1-boundaries namely the 1-boundaries of $\Sigma$ are generated by the boundaries of 2-cells of $\Sigma$ and the sum of all 2-cells has no boundary.
\begin{equation}\label{seqB1}
    \xymatrix{0\ar[r]&\ker\partial_2\ar[r]&\Csr_2(\Sigma)\ar[r]^{\partial_2}&B_1(\Sigma)\ar[r]&0}
\end{equation}
\begin{equation}\label{seqph1}
    \xymatrix{
    0\ar[r]&B_1(\Sigma)\ar[r] &Z_1(\Sigma)\ar[r]^-{p}& H
    _1(\Sigma)\ar[r]& 0
    }
\end{equation}
If we apply the left exact functor $\Hom_{\Zds\Isr} (E_o, -)$ to the above exact sequences, we can get the long exact sequences
\begin{equation}\label{eqn:HomEoP2sig}
    \xymatrix@C=1pc{0\ar[r]&\Hom_{ \Zds\Isr}(E_o,\Csr_2(\Sigma))\ar[r]^{\partial_{2,\ast}}&\Hom_{ \Zds\Isr}(E_o,B_1(\Sigma))\ar[r]&\Ext^{1}_{ \Zds\Isr}(E_o,\Zds)}.
\end{equation}
\begin{equation}\label{eqn:bacicexseqsig}
    \xymatrix@C=1pc{0\ar[r]&\Hom_{ \Zds\Isr}(E_o,B_1(\Sigma))\ar[r]&\Hom_{ \Zds\Isr}(E_o,Z_1(\Sigma))\ar[r]^-{p_*}&\Hom_{ \Zds\Isr}(E_o,H_1(\Sigma))}.
\end{equation}
Note that $\Hom_{\Zds\Isr}(E_o,\Zds)$ must be trivial. Unlike the $V$-case discussed in \cite{looijenga2021monodromy}, the first term of exact sequence \eqref{eqn:bacicexseqsig} is nontrivial. This could be seen from the character computation to the $\Cds\Isr$-module $\Csr_{2,\Cds}(\Sigma):=\Cds\otimes \Csr_2(\Sigma)$ which showed that $\Csr_{2,\Cds}(\Sigma)$ is isomorphic to $\Cds\oplus W\oplus I\oplus I'$ as $\Cds\Isr$-module, where $W$ is the $5$-dimensional irreducible representation.
Then the exact sequence \eqref{eqn:HomEoP2sig} will help to give a explicit description of $\Hom_{\Zds\Isr} (E_o, B_1(\Sigma))$.
For the term $\Hom_{ \Zds\Isr}(E_o,Z_1(\Sigma))$, we can "divide" $\Hom_{\Zds\Isr}(E_o,Z_1(\Sigma))$ into three parts, despite the one from the boundary, the other two parts are $\Hom_{\Zds\Isr}(E_o,Z_{\trc}(\Sigma))$ and $\Hom_{\Zds\Isr}(E_o,Z_{\edge}(\Sigma))$. We will discuss the last two separately. The main theorems of section is the Proposition \ref{OostrZ1}.

Let us first observe that the $\Zds$-module $\Csr_2(\Sigma)/(\iota-1)$ is isomorphic to $E_o$ as $\Zds\Isr$-module where the isomorphism is unique up to a sign and if we choose a system $\Rsr_2(\Sigma)$ of representatives of $\iota$-symmetry on $\Csr_2(\Sigma)$, $\Rsr_2(\Sigma)$ will become the $\Zds$-basis of $E_o$ that we discussed in Lemma \ref{lem:basisEo}.
Hence we may fix one such isomorphism and let $e$ denote not only one of the generators of $E_o$ but also a face in $\Rsr_2(\Sigma)\subset\Csr_2(\Sigma)$.
Moreover for each $z\in \Rsr_2(\Sigma)$, let $h_z\in\Acal_5$ be a permutation (which is not unique) such that $h_z z=-z$ and $\mu_5\cong\Stab(z)\subset \Isr$ be its $\Isr$-stabilizer.

To see this observation, recall that the $2$-cells i.e. the faces of $\Sigma$ can be canonically oriented clockwisely. Clearly they are bijectively indexed by the faces of $\tilde{\Sigma}$. The set $\Csr_2(\Sigma)$ of oriented 2-cells of $\Sigma$ admits both $\Isr$-symmetry and $\iota$-symmetry, Note that the $\Isr$-symmetry keeps the orientation and permutes the 12 faces, hence the stabilizer for each 2-cell is cyclic order five. The $\iota$-symmetry commutes with $\Isr$-action and will reverse the orientation.
Hence it is clear that the $\Zds$-module $\Csr_2(\Sigma)/(\iota-1)$ is isomorphic to $E_o$ as $\Zds\Isr$-module. If we choose a system $\Rsr_2(\Sigma)$ of representatives of $\iota$-symmetry on $\Csr_2(\Sigma)$, $\Rsr_2(\Sigma)$ will become the $\Zds$-basis of $E_o$ as we claimed above.
From the Corollary \ref{autgoEo}, such isomorphism is unique up to sign.

Let us first construct the morphisms $E_o\to\Csr_2(\Sigma)$ and $E_o\to B_1(\Sigma)$. Let $e\in\Rsr_2(\Sigma)$ be as above, the elements $\delta_{\cel}:=(e+\iota e)$ and $\delta'_{\cel}:=\sum_{z\in\Rsr_2(\Sigma)\backslash\{e\}}(z+\iota z)$ are $\Stab(e)$-invariant and $h_e\delta_{\cel}=-\delta_{\cel}$ resp. $h_e\delta'_{\cel}=-\delta'_{\cel}$. Then the elements $\delta_{\bound}:=\partial \sigma_{\cel}$ and $\delta'_{\bound}:=\frac{1}{2}\partial(\delta_{\cel}+\delta'_{\cel})$ are also $\Stab(e)$-invariant and signature reversal by $h_e$. Since $\delta'_{\bound}$ is the boundary of $\sum_{z\in\Rsr_2(\Sigma)}z$, the two elements $\delta_{\bound}$ and $\delta'_{\bound}$ both lie in $B_1(\Sigma)$.
Therefore we may define the $\Isr$-morphisms $E_o\to\Csr_2(\Sigma)$ by $\sigma_{\cel}:e\to \delta_{\cel}$ resp. $\sigma'_{\cel}:e\to \delta'_{\cel}$ and $\Isr$-morphisms $E_o\to B_1(\Sigma)$ by $\sigma_{\bound}:e\to \delta_{\bound}$ resp. $\sigma'_{\bound}:e\to \delta'_{\bound}$.


\begin{rmk}\label{nontrirep}
    We claim that $\sigma'_{\bound}\in\Hom_{\Zds\Isr}(E_o,B_1)$ is an element not coming from the image of $\Hom_{\Zds\Isr}(E_o,\Csr_2(\Sigma))$. If we assume the contrary that $\sigma'_{\bound}\in\Img\Hom_{\Zds\Isr}(E_o,\Csr_2(\Sigma))$, the $\Stab(e)$-invariance and $(h_e+1)$-invariance of $\sigma'_{\bound}(e)$ implies that $\sigma'_{\bound}$ must be the image of the following morphisms
    \begin{equation*}
        e \to \frac{1}{2}\sum_{z\in \Rsr_2(\Sigma)}(z+\iota z)
    \end{equation*}
    However it is clear that $\frac{1}{2}\sum_{z\in \Rsr_2(\Sigma)}(z+\iota z)$ is not an element of $\Csr_2(\Sigma)$. This is a contradiction!
\end{rmk}

We have proved in \cite{looijenga2021monodromy} that $Z_1(\Sigma)$ contains two disjoint factors $Z_{\trc}(\Sigma)$ and $Z_{\edge}(\Sigma)$. We will deal with the two factors separately and begin with the $E_o$-copy in $Z_{\trc}(\Sigma)$.
Recall that the 2-cell $e$ is a 10-gon which is modified from $\tilde{e}$, a regular pentagon on $\tilde{\Sigma}$. And $\tilde{e}$ has five vertices, the set of theses vertices is a $\Stab{(e)}$-orbit.
On the other hand, let $V_e\subset \Csr_0(\Sigma)$ be the set of terminal points of oriented edges $y\in \Csr_1(\tilde{\Sigma})$ who are not parallel to $\tilde{e}$ and have initial points on $\tilde{e}$ . This is a 5-element set and is a single $\Stab(e)$-orbit in $\Csr_0(\tilde{\Pi})$.
Note that despite the five vertices on $\tilde{e}$ and five vertices on $\iota\tilde{e}$, there are ten vertices of $\tilde{\Sigma}$ that do not lie on neither $\tilde{e}$ nor $\iota \tilde{e}$. They are the points of $V_e$ and $\iota V_e$.
Hence we may take the 5-elements sum $\delta_{\trc}:=\sum_{x\in \tilde{e}}\delta_x$ and $\delta'_{\trc}:=\sum_{x\in V_e}\delta_x$. They are $\Stab(e)$-invariant and satisfying that $h_e\delta_{\trc}=-\delta_{\trc}$ resp.$h_e\delta'_{\trc}=-\delta'_{\trc}$.
Therefore they define two $\Isr$-equivariant homomorphism $\sigma_{\trc}$ resp. $\sigma'_{\trc}$ from $E_o$ to $Z_{\trc}(\Sigma)$ by taking $\sigma_{\trc}(e):=\delta_{\trc}$ resp. $\sigma'_{\trc}(e):=\delta'_{\trc}$.

For the 2-cell $e$, the following two subsets of $\Csr_1(\tilde{\Sigma})$ have special interest to us
\begin{equation*}
    \begin{aligned}
        E_e  & := & \{y\in\Csr_1(\tilde{\Sigma}) & :\ini(y)\in e,\ \tm(y)\notin e        & \} \\
        E'_e & := & \{y\in\Csr_1(\tilde{\Sigma}) & :\ini(y)\in V_e,\ \tm(y)\in \iota V_e & \}
    \end{aligned}
\end{equation*}
Note that $E_e$ is a 5-elements set which is $\Stab{(e)}$-invariant and $E'_e$ is a $10$-elements set which is $\Stab(e)\times\iota$-invariant.
Each of the set defines a 1-cycle of edge type namely $\delta_{\edge}:=\sum_{y\in E_e}\delta_y$ and $\delta'_{\edge}:=\frac{1}{2}\sum_{y\in E'_e}\delta_y$. From the construction both of the elements are $\Stab(e)$ stable and signature reversal by $h_e$.
Hence the $\Isr$-orbit of $\delta_{\edge}$ resp. $\delta'_{\edge}$ has 12 elements and comes into 6 antipodal pairs.
Therefore each of the two elements defines an equivariant homomorphism $\sigma_{\edge}: E_o\to Z_{\edge}(\Sigma)$ resp. $\sigma'_{\edge}: E_o\to Z_{\edge}(\Sigma)$  with $\sigma_{\edge}(e)=\delta_{\edge}$ resp. $\sigma'_{\edge}(e)=\delta'_{\edge}$.

\begin{rmk}\label{Nodelta1}
    Let us consider the following sum in $Z_{\edge}(\Sigma)$
    \begin{equation*}
        \sum_{y\subset e\hbox{\ and oriented by $e$}}\delta_y
    \end{equation*}
    It is $\Stab(e)$-invariant. However from the properties of $\tilde{\Sigma}$, this element is $h_e$-invariant. Moreover the sum over the $\Isr$-orbit of this elment is zero. This is because each closed loop of edge type appears twice in this sum with opposite orientation. Hence this is a copy of $W_o$, the 5-dimensional permutation representation of $\Acal_5$.
\end{rmk}

\begin{prop}\label{SigC2B1ZeZtgen}
    The $\Zds$-modules $\Hom_{\Zds\Isr}(E_o,\Csr_2(\Sigma))$, $\Hom_{\Zds\Isr}(E_o,B_1(\Sigma))$, $\Hom_{\Zds\Isr}(E_o,Z_{\edge}(\Sigma))$ and $\Hom_{\Zds\Isr}(E_o,Z_{\trc}(\Sigma))$ are both free $\Zds$-modules of rank two. Moreover they are both $\Ocal_o$-modules where
    \begin{enumerate}
        \item $\Hom_{\Zds\Isr}(E_o,\Csr_2(\Sigma))$ is isomorphic to $\Ocal_o$ a free $\Ocal_o$-module of rank one with \begin{equation}\label{Xsigcel}
                  \begin{aligned}
                      X\sigma_{\cel}  & = & \sigma'_{\cel} \\
                      X\sigma'_{\cel} & = & 5\sigma_{\cel}
                  \end{aligned}
              \end{equation}
        \item $\Hom_{\Zds\Isr}(E_o,B_1(\Sigma))$ is isomorphic to $\Ocal$ with \begin{equation}\label{Xsigbound}
                  \begin{aligned}
                      X\sigma_{\bound}  & = & -\sigma_{\bound}+2\sigma'_{\bound} \\
                      X\sigma'_{\bound} & = & 2\sigma_{\bound}+\sigma'_{\bound}
                  \end{aligned}
              \end{equation} Hence it contains the image of $\Hom_{\Zds\Isr}(E_o,\Csr_2(\Sigma))$ as a submodule of index two,
        \item $\Hom_{\Zds\Isr}(E_o,Z_{\trc}(\Sigma))$ is a free $\Ocal_o$-module of rank one with\begin{equation}\label{Xsigtrc}
                  \begin{aligned}
                      X\sigma_{\trc}  & = & 2\sigma_{\trc}+\sigma'_{\trc} \\
                      X\sigma'_{\trc} & = & \sigma_{\trc}-2\sigma'_{\trc}
                  \end{aligned}
              \end{equation}
        \item $\Hom_{\Zds\Isr}(E_o,Z_{\edge}(\Sigma))$ is isomorphic to $\Ocal$ with \begin{equation}\label{Xsigedge}
                  \begin{aligned}
                      X\sigma_{\edge}  & = & \sigma_{\edge}+2\sigma'_{\edge} \\
                      X\sigma'_{\edge} & = & 2\sigma_{\edge}-\sigma'_{\edge}
                  \end{aligned}
              \end{equation}
    \end{enumerate}
\end{prop}

\begin{proof}
    Since $E_o$ is principal $\Zds\Isr$-module and $\Csr_2(\Sigma)$, $B_1(\Sigma)$, $Z_{\edge}(\Sigma)$ and $Z_{\trc}(\Sigma)$ are free $\Zds$-module, $\Hom_{\Zds\Isr}(E_o,\Csr_2(\Sigma))$, $\Hom_{\Zds\Isr}(E_o,B_1(\Sigma))$, $\Hom_{\Zds\Isr}(E_o,Z_{\edge}(\Sigma))$ and $\Hom_{\Zds\Isr}(E_o,Z_{\trc}(\Sigma))$ are both free $\Zds$-modules. The rank are clear from the $\Qds$-dimension of their $\Qds$-extension.

    \textbf{(Claim 1 and Claim 2):} It is clear from the computation that the Equations \eqref{Xsigcel} and \eqref{Xsigbound} hold. Hence $\Hom_{\Zds\Isr}(E_o,\Csr_2(\Sigma))$ is isomorphic to the module $\Ocal_o[\sigma_{\cel}]$, $\Hom_{\Zds\Isr}(E_o,B_1(\Sigma))$ is isomorphic to $\Ocal$ and the image of the module $\Hom_{\Zds\Isr}(E_o,\Csr_2(\Sigma))$ is contained in $\Hom_{\Zds\Isr}(E_o,B_1(\Sigma))$ as a submodule of index at least two. We claim that $\Hom_{\Zds\Isr}(E_o,B_1(\Sigma))$ is generated by $\sigma_{\bound}$ and $\sigma'_{\bound}$ as $\Zds$-module. Then all the assertions in the proposition implies from this claim. Let us assume the contrary that there exist a map $v\in\Hom_{\Zds\Isr}(E_o, B_1(\Sigma))$, such that $v$ is not generated by $\sigma_{\bound}$ and $\sigma'_{\bound}$ over $\Zds$. However we can find rational numbers $a$ and $a'$ with at least one of them is not an integer, such that $v=a\sigma_{\bound}+a'\sigma'_{\bound}$, since $\Hom_{\Zds\Isr}(E_o, B_1(\Sigma))$ is rank two and $\sigma_{\bound}$ and $\sigma'_{\bound}$ are not linearly equivalent. Then by counting the 1-cells of edge type on $\Sigma$, we find that both $a$ and $a'$ are integers, which is a contradiction.

    \textbf{(Claim 3 and Claim 4):} The direct computation shows that the Equations \eqref{Xsigedge} and \eqref{Xsigtrc} hold. Then from the construction and Remark \ref{Nodelta1}, $\Hom_{\Zds\Isr}(E_o,Z_{\trc}(\Sigma))$ is isomorphic to $\Ocal_o[\sigma_{\trc}]$ and $\Hom_{\Zds\Isr}(E_o,Z_{\edge}(\Sigma))$ is isomorphic to $\Ocal$.
\end{proof}

The following Lemma gave the intersection number between the class defined above and the vanishing cycles of the two degenerations.
Without loss of generality, Let $\Rsr_0(\tilde{\Sigma})$ be a systems of representatives of $\iota$-symmetry on $\Csr_0(\tilde{\Sigma})$ consists of the vertices of $e$ and the elements of $V_e$. And let $\Rsr_1(\tilde{\Sigma})$ be a system of representatives of $\iota\times(-1)$-action on $\Csr_1(\tilde{\Sigma})$ such that each element $y$ has initial point in $\Rsr_0(\tilde{\Sigma})$. Therefore the quantity of $\Rsr_0(\tilde{\Sigma})$ is 10 and $\Rsr_1(\tilde{\Sigma})$ is 15.

\begin{lem}\label{speint}
    Let $e$, $\Delta_{\edge}$, $\Delta_{\trc}$, $\sigma_{\trc}$, $\sigma'_{\trc}$, $\sigma_{\edge}$ and $\sigma'_{\edge}$ be as defined above. Then the class $[\sigma_{\edge}(e)]$ and $[\sigma'_{\edge}(e)]$ has zero intersection number with the elements of $\Delta_\edge$ resp.\  $[\sigma_{\trc}(e)]$ and $[\sigma'_{\trc}(e)]$ has zero intersection number with the elements of $\Delta_\trc$, whereas for $x\in \Rsr_0(\tilde{\Sigma})$ resp. $y\in \Rsr_1(\tilde{\Sigma})$,
    \begin{equation*}
        \begin{aligned}
             & \la [\sigma_{\edge}(e)],[\delta_x]\ra=
            \begin{cases} 1   & \text{if\ }  x\in e,  \\
              -1, & \text{if\ } x\in V_e. \\
            \end{cases}
            \\
             & \la [\sigma'_{\edge}(e)],[\delta_x]\ra=
            \begin{cases} 0  & \text{if  }  x\in e,  \\
              2, & \text{if }  x\in V_e.
            \end{cases}
            \\
             & \la [\sigma_{\trc}(e)],[\delta_y]\ra=
            \begin{cases}
                -1 & \text{if }\ini(y)\in e \hbox{ and }\tm(y)\in V_e, \\
                0, & \text{otherwise.}                                 \\
            \end{cases}
            \\
             & \la [\sigma'_{\trc}(e)],[\delta_y]\ra=
            \begin{cases}
                1  & \text{if }  \ini(y)\in e    \hbox{ and }\tm(y)\in V_e,      \\
                -2 & \text{if }  \ini(y)\in V_e \hbox{ and }\tm(y)\in \iota V_e, \\
                0, & \text{otherwise.}                                           \\
            \end{cases}
        \end{aligned}
    \end{equation*}
\end{lem}
\begin{proof}
    This is clear from the definitions (see also Figure \ref{trdo}).
\end{proof}

\begin{lem}
    The element $(\delta'_{\bound}+\delta'_{\edge}+\delta_{\trc}+\delta'_{\trc})$ is divisible by 2 in $Z_1(\Sigma)$. In particular the class $[\delta'_{\edge}+\delta_{\trc}+\delta'_{\trc}]$ is a boundary in $H_1(\Sigma,\Zds/2)$.
\end{lem}
\begin{proof}
    Let us $T_e$ to be the set of oriented 1-cells of truncation type on $\Sigma$ who are the intersection of $z\in\Rsr_2(\Sigma)$ and $\delta_x$ with $z\neq e$ and $x\in \iota V_e$ and the orientation inherits from $z$.
    It is clear to check that $(\delta'_{\bound}+\delta'_{\edge}+\delta_{\trc}+\delta'_{\trc})$ is a closed 1-cycle with even coefficients i.e. $2(\sum_{\ini(y)\in e,\tm(y)\in V_e}y+\sum_{y\in T_e}y)$.
\end{proof}

We are going to prove the following Proposition which describes the $\Zds\Isr$-module structure of $\Hom_{\Zds\Isr}(E_o,Z_1(\Sigma))$.
\begin{prop}\label{OostrZ1}
    Let us take the morphism $\sigma$ to be $\sigma:=\frac{1}{2}(\sigma'_{\bound}+\sigma'_{\edge}+\sigma_{\trc}+\sigma'_{\trc})$.
    In the exact sequence \eqref{eqn:bacicexseqsig} the cokernal of the map
    \begin{equation}
        \Hom_{\Zds\Isr}(E_o,Z_1(\Sigma))/\Hom_{\Zds\Isr}(E_o,B_1(\Sigma))
    \end{equation}
    is the free abelian group generated by $\sigma_{\trc}$,$\sigma'_{\trc}$, $\sigma_\edge$ and $\sigma$.
    The cokernel of $p_*$ is trivial.
\end{prop}
\begin{proof}
    We claim that the $\Zds$-module $\Hom_{\Zds\Isr}(E_o,Z_1(\Sigma))$ is free generated by $\sigma_{\bound}$, $\sigma'_{\bound}$, $\sigma_{\trc}$, $\sigma'_{\trc}$, $\sigma_{\edge}$ and $\sigma$. It is clear that the $\Qds$-dimension of $\Hom_{\Qds\Isr}(E_{\Qds},\Qds\otimes Z_1(\Sigma))$ is 6 and the 6 morphisms are not linearly equivalent. Hence if the claim didn't holds, there exist rational numbers $a_{\bound},a'_{\bound},\cdots,a_{\edge}$ and $a$
    with at least one of them is not an integer such that the following there linear combinations $a_{\bound}\sigma_{\bound}+\cdots+a\sigma $ lie in $Z_1(\Sigma)$
    The integrality of the coefficients of 1-cells of $\Pi$ implies that $a_{\bound},a'_{\bound},\cdots,a_{\edge}$ and $a$ are both integers. This is a contradiction!

    It is clear that from the Exact Sequence \eqref{eqn:bacicexseqsig}, the map
    \begin{equation}
        \Hom_{\Zds\Isr}(E_o,Z_1(\Sigma))/\Hom_{\Zds\Isr}(E_o,B_1(\Sigma))\to\Hom_{ \Zds\Isr}(E_o,H_1(\Sigma,\Zds))
    \end{equation}
    is injective. The image of $\sigma_{\bound}$ and $\sigma'_{\bound}$ vanishes. Hence we only need to prove that the image of $\sigma_{\trc}$, $\sigma'_{\trc}$, $\sigma_\edge$ and $\sigma$ generates the $\Zds$-module $\Hom_{ \Zds\Isr}(E_o,H_1(\Sigma,\Zds))$. We could check that the dimension of $\Hom_{\Qds\Isr}(E_o\otimes \Qds,H_1(\Sigma,\Qds))$ is four. And the image of $\sigma_{\trc}$, $\sigma'_{\trc}$, $\sigma_\edge$ and $\sigma$ are not linearly equivalent.Therefore we only need to check the conditions in Lemma \ref{genhom}.


    Now let $\alpha_1$, $\alpha_2$, $\alpha_3$ and $\alpha_4$ be distinct integers such that $\gcd(\alpha_1,\alpha_2,\alpha_3,\alpha_4)=1$. The morphism $s$ is taking as $s:=\alpha_1\sigma_{\trc}+\cdots+\alpha_4\sigma$.
    Let us assume $x_1$ and $x_2$ be elements of $\Rsr_0(\tilde{\Sigma})$ such that $x_1\in e$ and $x_2\in V_e$, $y_1$ and $y_2$ be the elements in $\Rsr_1(\tilde{\Sigma})$ such that $\ini(y_1)\in e$, $\tm(y_1)\in V_e$ and $\ini(y_1)\in V_e$, $\tm(y_1)\in \iota V_e$. Recall that we have discussed in Remark 3.2 of \cite{looijenga2021monodromy} that there exist a system $\{\gamma_x\}_{x\in\Rsr_0(\tilde{\Sigma})}$ of simple closed loops which is not necessarily $\Isr$-invariant such that $\{[\gamma_x],[\delta_x]\}_{x\in\Csr_0(\tilde{\Sigma})}$ become a basis for $H_1(\Sigma)$. And we have $\la[\gamma_x],[\delta_x]\ra$ equals 1 if and only if $x=x'$, otherwise it is always 0.
    Hence by the Lemma \ref{speint}, we have the following intersection numbers
    \begin{equation*}
        \begin{aligned}
            \la [s(e)],[\delta_{x_1}]\ra & = & \alpha_3                      \\
            \la [s(e)],[\delta_{x_2}]\ra & = & -\alpha_3+\alpha_4            \\
            \la [s(e)],[\delta_{y_1}]\ra & = & \alpha_2-\alpha_1             \\
            \la [s(e)],[\delta_{y_2}]\ra & = & \alpha_4-2\alpha_2            \\
            \la [s(e)],[\gamma_{x_1}]\ra & = & \alpha_1+N(\alpha_3,\alpha_4) \\
        \end{aligned}
    \end{equation*}
    Here $N(\alpha_3,\alpha_4)$ is a integral linear combination of $\alpha_3$ and $\alpha_4$. We could check that the five numbers $\alpha_3$, $-\alpha_3+\alpha_4$, $\alpha_2-\alpha_1$, $\alpha_4-2\alpha_2$ and $\alpha_1+N(\alpha_3,\alpha_4)$ are coprime.
    Hence there exist coprime numbers $\{\beta_i\}_{i=1}^5$ such that $\beta_1\alpha_3+\beta_2(-\alpha_3+\alpha_4)+\beta_3(\alpha_2-\alpha_1)+\beta_4(\alpha_4-2\alpha_2)+\beta_5(\alpha_1+N(\alpha_3,\alpha_4))=1$.
    Then let us take $y(\alpha_1,\alpha_2,\alpha_3,\alpha_4)$ to be the combination $\beta_1[\delta_{x_1}]+\beta_2[\delta_{x_2}]+\beta_3[\delta_{y_1}]+\beta_4[\delta_{y_2}]+\beta_5[\gamma_{x_1}]$. It is clear that $\la[s(e)],[y(\alpha_1,\alpha_2,\alpha_3,\alpha_4)]\ra=1$.
    Hence the Lemma \ref{genhom} above implies the second assertion. This finishes the proof.
\end{proof}
\begin{cor}\label{Cor:SiggenSL2O}
    Let $p$ is the natural map $Z_1(\Sigma)\to H_1(\Sigma)$ as above, the $\Isr$-equivariant morphism $U_{\Sigma}$ and $V_{\Sigma}$ in $\Hom_{\Zds\Isr}(E_o,H_1(\Sigma))$ be defined as $U_{\Sigma}:=p\circ \sigma_{\trc}$ and $V_{\Sigma}:=p\circ \sigma$. The $\Ocal_o$-module $\Hom_{\Zds\Isr}(E_o,H_1(\Sigma))$ is a free $\Ocal_o$-module of rank two with generators $U_{\Sigma}$ and $V_{\Sigma}$.
\end{cor}
\begin{proof}
    We have proved in Proposition \ref{OostrZ1}, the $\Zds$-module $\Hom_{\Zds\Isr}(E_o,H_1(\Sigma))$ is freely generated by the image of $\sigma_{\trc}$, $\sigma'_{\trc}$, $\sigma'_{\edge}$ and $\sigma$.
    Since the images of $\sigma_{\bound}$ and $\sigma'_{\bound}$ vanishes in $H_1(\Sigma)$, we have $V_{\Sigma}=\frac{1}{2}(p\circ\sigma'_{\edge}+p\circ\sigma_{\trc}+p\circ\sigma'_{\trc})$. Combined with the Equations \eqref{Xsigtrc} and \eqref{Xsigedge}, we have the following
    \begin{equation}\label{EQ:SiggenSL2O}
        \begin{aligned}
             & p\circ\sigma'_{\trc}  & = & (X-2)U_{\Sigma}              \\
             & p\circ\sigma_{\edge}  & = & -2U_{\Sigma}+(X+1)V_{\Sigma} \\
             & p\circ\sigma'_{\edge} & = & (1-X)U_{\Sigma}+2V_{\Sigma}
        \end{aligned}
    \end{equation}
    Moreover it is clear to check that $U_{\Sigma}$ and $V_{\Sigma}$ are not linearly equivalent over $\Ocal_o$. These facts imply the Corollary.
\end{proof}


\section{The Geometric Model from the Bring's Curve}
In this section we will introduce a geometric model for the smooth fiber $C_t$ where $t\in \Bsr^{\circ}$ constructed form the geometric model of the Bring's Curve. We will describe two stable degeneration given by this model and give explicit descriptions of the vanishing cycle of the degeneration. At beginning we will introduce some properties of the Bring's curve that is required below. For more detailed introduction to the Bring's curve, we refer to the survey paper \cite{braden2022bring} by H.\ Braden and L.\ Disney-Hogg.

\subsection{Preliminaries on the Bring's Curve}

Let us begin with a regular icosahedron in $I_{\Rds}$ centered at the origin where $I_{\Rds}$ is a fixed Euclidean 3-space with $\Acal_5$-symmetry as the section \ref{seclatt}. It has 12 vertices, 30 edges and 20 faces. Moreover the antipodal map $\iota$ is well-defined as above. Then a great dodecahedron $\tilde{\Pi}$ has the same edges and vertices as the icosahedron above. However its faces are replaced by inscribed planar regular pentagons that connects 5 coplanar vertices. Hence the number of faces is 12. Every face $z$ has a unique parallel face which is by construction $\iota z$. Note that if $z$ and $z'$ are two different faces that are not parallel, they will intersect on the edges or the vertices and no where else. If we denote the $i$-th cells of $\tilde{\Pi}$ by $\Csr_{i}(\tilde{\Pi})$ with $i=0,1,2$ similar as the last section, the numbers of each set is $12$, $30$ and $12$.
Hence the Euler formula gives
\begin{equation*}
    \sharp\Csr_{0}(\tilde{\Pi})-\sharp\Csr_{1}(\tilde{\Pi})+\sharp\Csr_{2}(\tilde{\Pi})=-6
\end{equation*}
which is the Euler characteristic of a genus 4 surface. It is actually a complex algebraic curve of genus 4 with at least $\Acal_5$-symmetry, for the flat structure on each face can glue together making it a locally flat surface with the vertices as the singularities. Also remember that it was proved in \cite{Cheltsov2015Cremona} that Bring's curve is the only non-hyperelliptic genus 4 curve with $\Acal_5$-symmetry. In other words $\tilde{\Pi}$ with the complex structure above is isomorphic to the Bring's curve.

\begin{rmk}
    Let us consider the map projection away from the origin to the circumscribed icosahedron. We could endow a flat structure on each face of the icosahedron and make it a Riemann surface isomorphic to $\Pds^1$. Then this map will become a ramified triple cover branched at the vertices of the icosahedron. The ramification index is 2 at the vertices of $\tilde{\Pi}$ and 1 at the center of the faces. We can check that these data make the Riemann-Hurwitz formula holds. Hence this map realized the Bring's curve as the branched triple cover of $\Pds^1$ which is not $\Isr$-equivariant. This gives another way of making $\tilde{\Pi}$ a complex algebraic curve by pulling back the complex structure on $\Pds^1$ through this triple cover.
\end{rmk}

\subsection{The Geometric Model $\Pi$ and its Degenerations}
\begin{figure}\label{gtdode}
    \centering
    {\includegraphics[width=0.2\textwidth]{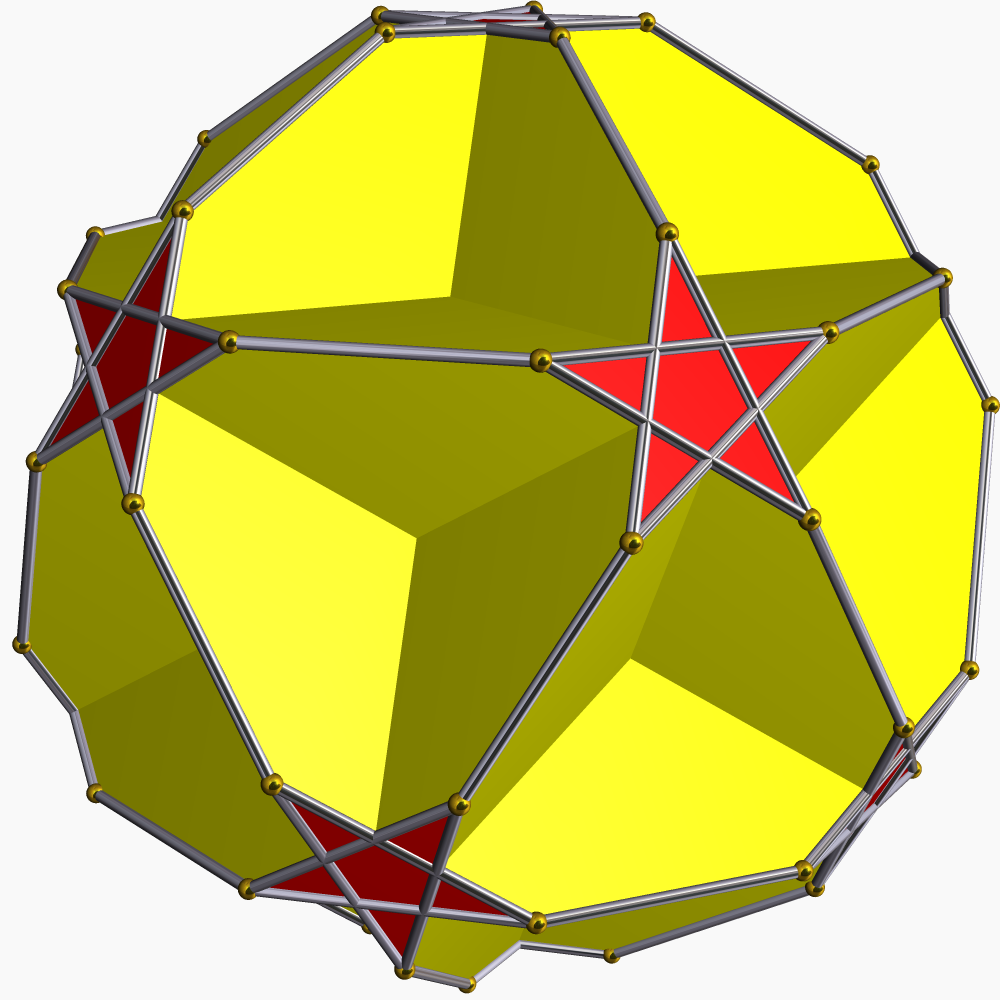}}
    \caption{\small{Removing in an $\Isr$-equivariant manner a small regular penta-pyramid at each vertex of $\tilde{\Pi}$}}
\end{figure}
Next we will modify the model $\tilde{\Pi}$ of the Bring's curve to get a model $\Pi$ of genus 10 curve and we will introduce two degenerations coming from this model.
Let $\hat{\Pi}$ be obtained from $\tilde{\Pi}$ by removing in an $\Isr$-equivariant manner a small regular penta-pyramid at each vertex $x\in \Csr_0(\tilde{\Pi})$. Then the boundary of $\hat{\Pi}$ consists of 12 disjoint closed loops, each of them is a regular pentagram centered at $x$. Note that this operation is the same as following: for each face $z\in\Csr_2(\tilde{\Pi})$, we will remove a small isosceles triangles in $\Stab{(z)}$-manner at each vertex of $z$. Hence each face of $\hat{\Pi}$ is a decagon.
We will identify the opposite points on the boundary of $\hat{\Pi}$ and thus obtained a complex $\Pi$ that is a closed oriented surface of genus 10 endowed with an $\Isr$-symmetry.
It is clear that $\Pi$ has a structure of cellular decomposition: the set of 2-cells of $\Pi$ consists of 12 decagons and is indexed by $\Csr_2(\tilde{\Pi})$ i.e. the 2-cells of the great dodecahedron. The set of 0-cells of $\Pi$ are represented by the antipodal pairs of vertices of $\hat{\Pi}$ and are naturally indexed by the oriented 1-cells of $\tilde{\Pi}$.
The set of 1-cells of $\Pi$ consists two disjoint subset: those lie on the edge of $\tilde{\Pi}$ (hence called of edge type and denoted as $\Csr_{\edge}({\Pi})$) and those come from the boundary of $\hat{\Pi}$ and they form 6 antipodal pairs of pentagrams (hence called of truncation type and denoted as $\Csr_{\trc}({\Pi})$).

\begin{prop}
    The action of $\Isr$ on the cells of $\Pi$ is as following:
    \begin{enumerate}
        \item the action of $\Isr$ on the set $\Csr_0({\Pi})$ of 0-cells of ${\Pi}$ is transitive, each 0-cell has a stabilizer cyclic of order five,
        \item the set $\Csr_1({\Pi})$ of oriented 1-cells of ${\Pi}$ consists of two regular orbits $\Csr_{\edge}({\Pi})$ and $\Csr_{\trc}({\Pi})$,
        \item the action of $\Isr$ on the set $\Csr^{+}_2(\Pi)$ of canonically oriented 2-cells is transitive, the stabilizer of each such cell is cyclic of order five.
    \end{enumerate}
\end{prop}
An oriented 1-cell of $\Pi$ is part of a unique loop consisting of oriented cells of the same type. We will analysis this in detail. A loop of truncation type consists of 5 oriented 1-cells of that type and each oriented 1-cell of truncation type appears in a unique such loop. They are indexed by the set $\Csr_0(\tilde{\Pi})$: every vertex $x$ of $\tilde{\Pi}$ lies in the center of a solid regular pentagram whose interior is removed to form $\hat{\Pi}$. The boundary of this pentagram together with its counterclockwise orientation is a sum $\theta_x$ of five oriented 1-cells of truncation type.
Moreover we have that $\theta_x=-\theta_{\iota x}$. We will call the closed loop constructed above loops of the truncation type. There are 12 such closed loops and the $\Isr$-action permutes them transitively. The stabilizer of each closed loop of truncation type is cyclic of order 5. We will denote the set of such twelve 1-cycles by $\Theta_{\trc}$.

A loop of edge type is the sum of 1-cells of that type and its image under $-\iota$. They are bijectively indexed by the set $\Csr_1(\tilde{\Pi})$and we will denote the oriented loop of edge type defined by $y\in \Csr_1(\tilde{\Pi})$ by $\theta_y$. It is clear that we have $\theta_{\iota y}=-\theta_y$ and $\theta_{-y}=-\theta_y$. The set of such 1-cycles, which we will denote it by $\Theta_{\edge}$, is an $\Isr$-orbit consists of 30 elements. In other words, the $\Isr$-stabilizer of each 1-cycle of edge type is cyclic of order two.
The following Lemma gives the intersection number between the two kinds of 1-cycles and is straightforward to check.
\begin{lem}\label{intPi}
    The intersection numbers of these $1$-cycles are as follows: any two loops of the same type have intersection number zero and
    if
    $x\in \Csr_0(\tilde{\Pi})$ and $y\in \Csr_1(\tilde{\Pi})$, then $\la \theta_x, \theta_y\ra =0$ unless $x$ lies on $y$ or on
    $\iota y$, in which case $\la \theta_x, \theta_y\ra\in\{\pm 1\}$ with the plus sign appearing if and only if $x$ is the end point of $y$.
\end{lem}


Let us describe two kinds of degenerations realized by the model $\Pi$, they are both genus 10 nodal curves with $\Isr$-symmetry.
They will have the properties that one has $\Theta_{\edge}$ as vanishing cycles and the other one has $\Theta_{\trc}$ as vanishing cycles.

First let us note that there exist a one-parameter family of piecewise Euclidean structure on $\hat{\Pi}$. It is given as following: let us assume the length of the 1-cells of the edge type is $\tau>0$ and the length of the 1-cells of the truncation type is $1-\tau$. This determines $\hat{\Pi}$ as a metric space. It is clear that this metric is piecewise Euclidean and invariant under both $\Isr$ and $\iota$-symmetry. It defines a conformal structure $I_{\tau}$ first on $\Pi$ except the vertices and then it can be extended across these vertices. The given orientation on $\Pi$ makes this conformal structure an $\Isr$-invariant complex structure.

If $\tau$ tends to 1, we got a complex structure on the singular surface $\Pi_{\trc}$ that is obtained from $\Pi$ by contracting each 1-cycle of truncation type into a point. It is clear that this singular surface can also be obtained by identifying the 6 pairs in $\Csr_{0}(\tilde{\Pi})$. The complex structure $I_{1}$ makes it a singular curve isomorphic to $C_{-1}$.
Similarly, if $\tau$ tends to $0$, the length of 1-cycles of edge type will tend to 0. We got a complex structure on the singular surface $\Pi_{\edge}$ that is obtained from $\Pi$ by contracting each 1-cycle of edge type into a point. The complex structure $I_{0}$ makes it a singular curve isomorphic to $C_{\infty}$.
Summarize above, we got the following Proposition.
\begin{prop}\label{hypeModV}
    The Riemann surface $(\Pi, I_\tau)$ is the set of complex points of a complex real algebraic curve. It has genus 10 and comes with a faithful $\Isr$-action, hence is isomorphic to a member of the Winger pencil.
    We thus have defined a continuous map $\gamma_{\Pi}: [0,1]\to \Bsr$ which transverse $(0,1)$ to $\Bsr^{\circ}$ and $\tau=1$ to $t=-1\in \Bsr$ resp.  $\tau=0$ to $t=\infty\in \Bsr$
    such that the pull-back of the Winger pencil yields the family constructed above. The degenerations of $(\Pi, I_\tau)$ into $\Pi_\trc$ resp. $\Pi_{\edge}$ have $\Theta_\trc$ resp. $\Theta_{\edge}$ as its set of vanishing cycles.
\end{prop}
\begin{rmk}
    The polygon $\Pi_{\edge}$ has interesting properties itself. If we treat each $\theta_x$ as a solid regular pentagon not as a closed loop, the resulting polygon is named as a \textit{dodecadodecahedron}. This polygon has 24 faces, 12 of them are regular pentagons and 12 of them are regular pentagrams, 60 edges and 30 vertices, giving the Euler characteristic $\chi=-6$.
    It was showed in \cite{weber2005kepler} that this is also an Euclidean realization for the Bring's curve.
\end{rmk}

\subsection{Cellular Homology of $\Pi$}
The geometric model $\Pi$ admits a cellular structure which enables us to compute its homology as the homology of the combinatorial chain complex
\begin{equation}\label{seqLam}
    \xymatrix{
        0\ar[r]&\Csr_2(\Pi)\ar[r]^{\partial_2} &\Csr_1(\Pi)\ar[r]^{\partial_1}& \Csr_0(\Pi)\ar[r]& 0.
    }
\end{equation}
Note that the middle term $\Csr_1(\Pi)$ admits a direct sum decomposition namely $\Csr_1(\Pi)=\Csr_{\edge}\oplus\Csr_{\trc}$. Similar as above we will denote the set of $i$-cycles as $Z_i(\Pi):=\ker(\partial_i)$ and $i$-boundaries as $B_i(\Pi):=\Img(\partial_{i+1})$.
Let us apply the functor $\Hom_{\Zds\Isr(E_o,-)}$ to the Exact Sequence \eqref{seqB1} and \eqref{seqph1} with $\Sigma$ replaced by $\Pi$.
For the first one we have the long exact sequence
\begin{equation}\label{eqn:C2B1Pi}
    \xymatrix@C=1pc{0\ar[r]&\Hom_{ \Zds\Isr}(E_o,\ker\partial_2)\ar[r]&\Hom_{ \Zds\Isr}(E_o,\Csr_2(\Pi))\ar[r]&\Hom_{ \Zds\Isr}(E_o,B_1(\Pi))}.
\end{equation}
It is from the construction of $\Pi$ that $\ker \partial_2$ is isomorphic to trivial representation of $\Isr$. Hence the first term vanishes. We could see from the character computation that $\Csr_2(\Pi)\otimes \Cds$ is isomorphic to $\Id\oplus W\oplus I\oplus I'$. Hence the second and the third terms are nontrivial. We will see below that one of them is isomorphic to $\Ocal_o$ and another is isomorphic to $\Ocal$ as $\Ocal_o$-modules and the quotient of them is finite but nonzero.
For the second we have the following
\begin{equation}\label{eqn:bacicexseqPi}
    \xymatrix@C=1pc{0\ar[r]&\Hom_{ \Zds\Isr}(E_o,B_1(\Pi))\ar[r]^{i_*}&\Hom_{ \Zds\Isr}(E_o,Z_1(\Pi))\ar[r]^-{p_*}&\Hom_{ \Zds\Isr}(E_o,H_1(\Pi))}.
\end{equation}
We will introduce four elements of $\Hom_{ \Zds\Isr}(E_o,Z_1(\Pi))$ with two of them have image in the $\Zds$-module spanned by $\Theta_{\edge}$ (which we will denote it as $Z_{\edge}(\Pi)$) denoted by $\pi_{\edge}$ and the other two $\pi_{\trc}$ and $\pi'_{\trc}$ in $\Zds$-module spanned by $\Theta_{\trc}$ (which we will denote it as $Z_{\trc}(\Pi)$).

Let us first observe that the $\Zds\Isr$-module $\Csr_0(\tilde{\Pi})/(\iota+1)$ is isomorphic to $E_o$ where this isomorphism is unique up to a sign and there exist a system of representatives $\Rsr_0(\tilde{\Pi})$ of $\iota$-action on $\Csr_0(\tilde{\Pi})$ such that this isomorphism will identify this system with the basis $\{ e, e_0,\cdots, e_4\}$ of $E_o$. To see this observation recall that $\Csr_0(\tilde{\Pi})$ consists of 12 vertices where the $\Isr$-symmetry permutes them making the set one $\Isr$-orbit. The map $\iota$ commutes with this $\Isr$-symmetry. Moreover the uniqueness comes from the Lemma \ref{autgoEo}.
We will take $e$ represent not only the element in the basis of $E_o$ but also a vertex in $\Rsr_{0}(\tilde{\Pi})$. Finally for each $x\in\Rsr_0(\tilde{\Pi})$, let us fix an element $h_x\in\Isr$ cyclic of order two, which is not necessarily unique, such that $h_xx=-x$.

Let us begin with the modules $\Hom_{\Zds\Isr}(E_o,\Csr_2(\Pi))$ and $\Hom_{\Zds\Isr}(E_o,B_1(\Pi))$.
For arbitrary vertex $x\in\Rsr_0(\tilde{\Pi})$, there exist two parallel planer pentagons such that the vertices of $\tilde{\Pi}$ despite $x$ and $\iota x$ lie on one of them. We will denote the two planer pentagon together with its counterclockwise orientation by $z_x$ resp. $z_{\iota x}$. Clearly each of them associates to a oriented 2-cell in $\Csr_2(\Pi)$ in a natural way, we will denote the two 2-cells in the same symbols. It is clear to check that $z_{\iota x}=z_{h_x x}=h_x z_x=-\iota z_x$.

From this observation there are two elements in $\Csr_2(\Pi)$ that draw our attention namley $\theta_{\cel}:=z_e+\iota z_e$ and $\theta'_{\cel}:=\sum_{x\in \Rsr_0(\tilde{\Pi})\backslash\{e\}}(z_x+\iota z_x)$.
They are $\Stab(e)$-invariant and satisfies the property that $h_e\theta_{\cel}=-\theta_{\cel}$ resp.$h_e\theta'_{\cel}=-\theta'_{\cel}$. The two elements give two $\Isr$-equivariant morphisms $\pi_{\cel}$ resp. $\pi'_{\cel}$ of $E_o\to \Csr_2(\Pi)$, namely $e\to\theta_{\cel}$ resp. $e\to \theta'_{\cel}$.


The boundaries of $\theta_{\cel}$ resp. $\theta'_{\cel}$ satisfies the relation that $\partial_2\theta'_{\cel}=\partial_2\theta_{\cel}+2\partial_2(-\iota z_e+\sum_{x\in \Rsr_0(\tilde{\Pi})\backslash\{e\}}z_x)$.
Let us take the two elements $\theta_{\bound}$ resp. $\theta'_{\bound}$ as $\theta_{\bound}:=\partial_2\theta_{\cel}$ and $\theta'_{\bound}:=\partial_2(-\iota z_e+\sum_{x\in \Rsr_0(\tilde{\Pi})\backslash\{e\}}z_x)$.
Since the element $\theta'_{\bound}$ is the boundary of $\frac{1}{2}(\theta'_{\cel}-\theta_{\cel})$, it has to be $\Stab(e)$-equivariant and satisfy the relation $h_e\theta'_{\bound}=-\theta'_{\bound}$. It is the same for $\theta_{\bound}$. Hence we have two $\Isr$-equivariant morphisms in $\Hom_{\Zds\Isr}(E_o,B_1(\Pi))$ i.e. $\pi_{\bound}$ resp. $\pi'_{\bound}$ given as $e\to\theta_{\bound}$ resp. $e\to \theta'_{\bound}$.

\begin{rmk}
    Observe that the element $(-\iota z_e+\sum_{x\in \Rsr_0(\tilde{\Pi})\backslash\{e\}}z_x)\in \Csr_2(\Pi)$ is $\Stab(e)$-invariant. However instead of inverting its signature, the element $h_e$ will fix it. These facts implies that $\pi'_{\bound}$ is an element of $\Hom_{\Zds\Isr}(E_o,B_1(\Pi))$ but it does not lie in the image of $\partial_2$.
\end{rmk}
Next let us consider the $\Zds\Isr$-modules $\Hom_{\Zds\Isr}(E_o,Z_{\trc}(\Pi))$ and $\Hom_{\Zds\Isr}(E_o,Z_{\edge}(\Pi))$. For the $Z_{\trc}(\Pi)$-part,
let us take $\theta_{\trc}=\theta_e$ and $\theta'_{\trc}:=\sum_{x\in \Rsr_0(\tilde{\Pi})\backslash\{e\}}\theta_x$. It is clear that they are $\Stab(e)$-invariant and signature reversal by $h_e$.
Therefore we may define the map $E_o\to Z_{\trc}$ in an $\Isr$-equivariant manner $\pi_{\trc}$ to be the morphism as $e\to \theta_e$ resp. $\pi'_{\trc}:e\to\theta'_{\trc}$.


For the $Z_{\edge}(\Pi)$-part, observe that for each vertex $x\in\Rsr_0(\tilde{\Pi})$ there exist five oriented 1-cells $y\in\Csr_1(\tilde{\Pi})$ such that they have $x$ as common initial point. Besides of these edges and their $\iota$-dual, the oriented edges that don't parallel to $z_e$ is the 10-element subset $\{y\in \Csr_1(\tilde{\Pi}):\ini(y)\in z_e, \tm(y)\in z_{\iota_e}\}$ of $\Csr_1(\tilde{\Pi})$ that admits the symmetry of $\Stab(e)\times(-\iota)$. This set consists of two $\Stab(e)$-orbit and $-\iota$ exchanges the two orbits.
From these observations, let us take $\theta_{\edge}:=\sum_{y\in\Csr_1(\tilde{\Pi}),\ini(y)=e} \theta_y$ and $\theta'_{\edge}:=\frac{1}{2}\sum_{y\in \Csr_1(\tilde{\Pi}),\ini(y)\in z_e,\tm(y)\in z_{\iota e}}\theta_y$. Since the 1-cycle of edge type satisfies the relation that $\theta_{-\iota y}=\theta_{y}$, the element $\theta'_{\edge}$ lies in $Z_{\edge}(\Pi)$. Moreover the two elements are both $\Stab{(e)}$-invariant and signature reversal by $h_e\in\Isr$.
Hence we can define the morphism $\pi_{\edge}$ resp. $\pi'_{\edge}$ to be the $\Isr$-equivariant map $e\to \theta_{\edge}$ resp. $e\to \theta'_{\edge}$.

\begin{rmk}
    The element $\sum_{y\subset \partial z_e}\theta_y$, where $y$ is naturally oriented such that it is the same as the boundary of $z_e$, is also $\Stab(e)$-invariant. However $h_e$ will fix this element. Hence it cannot give a morphism from $E_o$ to $Z_{\edge}(\Pi)$.
\end{rmk}
We have the following Propositions.
\begin{prop}\label{Zbodpgmgen}
    The $\Zds$-modules $\Hom_{\Zds\Isr}(E_o,\Csr_2(\Pi))$, $\Hom_{\Zds\Isr}(E_o,B_1(\Pi))$, $\Hom_{\Zds\Isr}(E_o,Z_{\edge}(\Pi))$ and $\Hom_{\Zds\Isr}(E_o,Z_{\trc}(\Pi))$ are both free $\Zds$ of rank two. Moreover they are both $\Ocal_o$-modules where
    \begin{enumerate}
        \item $\Hom_{\Zds\Isr}(E_o,\Csr_2(\Pi))$ is a free $\Ocal_o$-module of rank one with\begin{equation}\label{Xpicel}
                  \begin{aligned}
                      X\pi_{\cel}  & = & \pi'_{\cel} \\
                      X\pi'_{\cel} & = & 5\pi_{\cel}
                  \end{aligned}
              \end{equation}
        \item $\Hom_{\Zds\Isr}(E_o,B_1(\Pi))$ is isomorphic to $\Ocal$ as $\Ocal_o$-modules with \begin{equation}\label{Xpibound}
                  \begin{aligned}
                      X\pi_{\bound}  & = & \pi_{\bound}+2\pi'_{\bound} \\
                      X\pi'_{\bound} & = & 2\pi_{\bound}-\pi'_{\bound}
                  \end{aligned}
              \end{equation} Moreover it contains the image of $\Hom_{\Zds\Isr}(E_o,\Csr_2(\Pi))$ as a submodule of index two.
        \item $\Hom_{\Zds\Isr}(E_o,Z_{\trc}(\Pi))$ is a free $\Ocal_o$-module of rank one with\begin{equation}\label{Xpitrc}
                  \begin{aligned}
                      X\pi_{\trc}  & =  \pi'_{\trc} \\
                      X\pi'_{\trc} & =  5\pi_{\trc}
                  \end{aligned}
              \end{equation}
        \item $\Hom_{\Zds\Isr}(E_o,Z_{\edge}(\Pi))$ is isomorphic to $\Ocal$ as $\Ocal_o$-module with \begin{equation}\label{Xpiedge}
                  \begin{aligned}
                      X\pi_{\edge}  & = & -\pi_{\edge}  +  2\pi'_{\edge} \\
                      X\pi'_{\edge} & = & 2\pi_{\edge} +  \pi'_{\edge}
                  \end{aligned}
              \end{equation}
    \end{enumerate}
\end{prop}
\begin{proof}
    The proof for this Proposition is similar to the proof of Proposition \ref{SigC2B1ZeZtgen}. The Equations \ref{Xpicel}, \ref{Xpibound}, \ref{Xpitrc} and \ref{Xpiedge} come from direct computation. And the other claims come from counting on the coefficients of 1-cells.
\end{proof}

\begin{lem}
    The element $(\theta'_{\bound}+\theta_{\trc}-\theta'_{\trc}+\theta'_{\edge})$ is divisible by two in $Z_1(\Pi)$. In particular, the class $[\theta_{\trc}-\theta'_{\trc}+\theta'_{\edge}]$ is a boundary in $H_1(\Pi,\Zds/2)$.
\end{lem}
\begin{proof}
    This is a direct computation.
\end{proof}
\begin{prop}\label{Z1Pigen}
    Let us take the morphism $\pi$ to be $\pi:=\frac{1}{2}(\pi'_{\bound}+\pi_{\trc}-\pi'_{\trc}+\pi'_{\edge})$. Then
    in the Exact Sequence \eqref{eqn:bacicexseqPi}, the cokernel of the map $i_{*}$
    \begin{equation*}
        \Hom_{\Zds\Isr}(E_o,Z_1(\Pi))/ \Hom_{\Zds\Isr}(E_o,B_1(\Pi))
    \end{equation*}
    is free abelian group generated by $\pi_{\trc}$, $\pi'_{\trc}$,$\pi_{\edge}$ and $\pi$. The cokernel of $p_{*}$ is trivial.
\end{prop}
Before we give the proof of the Proposition \ref{Z1Pigen}, let us given the intersection numbers of some cycle class.
\begin{prop}\label{intedtrPi}
    Let $e$, $\Theta_{\edge}$, $\Theta_{\trc}$, $\pi_{\edge}$, $\pi'_{\edge}$, $\pi_{\trc}$ and $\pi'_{\trc}$ be defined as before. Then the classes $[\pi_{\edge}(e)]$ and $[\pi'_{\edge}(e)]$ resp. $[\pi_{\trc}(e)]$ and $[\pi'_{\trc}(e)]$ has zero intersection number with the elements in $\Theta_{\edge}$ resp. $\Theta_{\trc}$. Meanwhile for $x\in \Rsr_0(\tilde{\Pi})$ and $y\in \Csr_{1}(\tilde{\Pi})$ with $\ini(y)\in \Rsr_0(\tilde{\Pi})$, we have
    \begin{equation*}
        \begin{aligned}
             & \la [\pi_{\edge}(e)],[\theta_x]\ra=
            \begin{cases} 5   & \text{if }  x=e,    \\
              -1, & \text{if\ } x\neq e \\
            \end{cases}
            \\
             & \la [\pi'_{\edge}(e)],[\theta_x]\ra=
            \begin{cases} 0  & \text{if }  x=e,     \\
              2, & \text{if }  x\neq e.
            \end{cases}
            \\
             & \la [\pi_{\trc}(e)],[\theta_y]\ra=
            \begin{cases} -1 & \text{if } \ini(y)=e,         \\
              0, & \text{if }  \text{otherwise.} \\
            \end{cases}
            \\
             & \la [\pi'_{\trc}(e)],[\theta_y]\ra=
            \begin{cases}
                1  & \text{if } \ini(y)=e,                                            \\
                -2 & \text{if }  \ini(y)\in z_e \text{\ and\ } \tm(y)\in z_{\iota e}, \\
                0, & \text{if }  \text{otherwise.}                                    \\
            \end{cases}
        \end{aligned}
    \end{equation*}
\end{prop}
\begin{proof}
    This is a direct compute from the model $\Pi$.
\end{proof}
We have seen on the model of $\Sigma$, each $\delta_x$ with $x\in\Csr_0(\tilde{\Sigma})$ admits a "dual" class such that they span $H_1(\Sigma)$ together. The similar construction can be made for the model $\Pi$. However they will only span a primitive sublattice of $H_1(\Pi)$.
\begin{prop}\label{dulintPi}
    For each vertex $x\in \Rsr_0{(\tilde{\Pi})}$, there exist a 1-cycle $[\varepsilon_x]$ such that $\la [\varepsilon_x],\theta_{x'}\ra=1$ if and only if $x'=x$ and otherwise it is $0$ for all $x'\in\Rsr_{0}(\tilde{\Pi})$. In particular, if $x\neq e$ we could require additional conditions for $\varepsilon_x$ such that
    \begin{enumerate}
        \item $\la [\pi_{\edge}(e)],[\varepsilon_x]\ra=0$ and
        \item $\la [\pi'_{\edge}(e)],[\varepsilon_x]\ra=-1$.
    \end{enumerate}
\end{prop}
\begin{proof}
    The proof for the first claim is clear. The construction above made the $\Pi$ a genus 10 Riemann surface and each of the six loops of truncation type represents a generators of $\pi_1(\Sigma)$ which is canonical. Hence their exist "dual" class $[\varepsilon_x]$ such that $\la [\varepsilon_x],\theta_{x'}\ra=1$ if and only if $x'=x$.
    \begin{figure}\label{intdode}
        \centering
        {\includegraphics[width=0.2\textwidth]{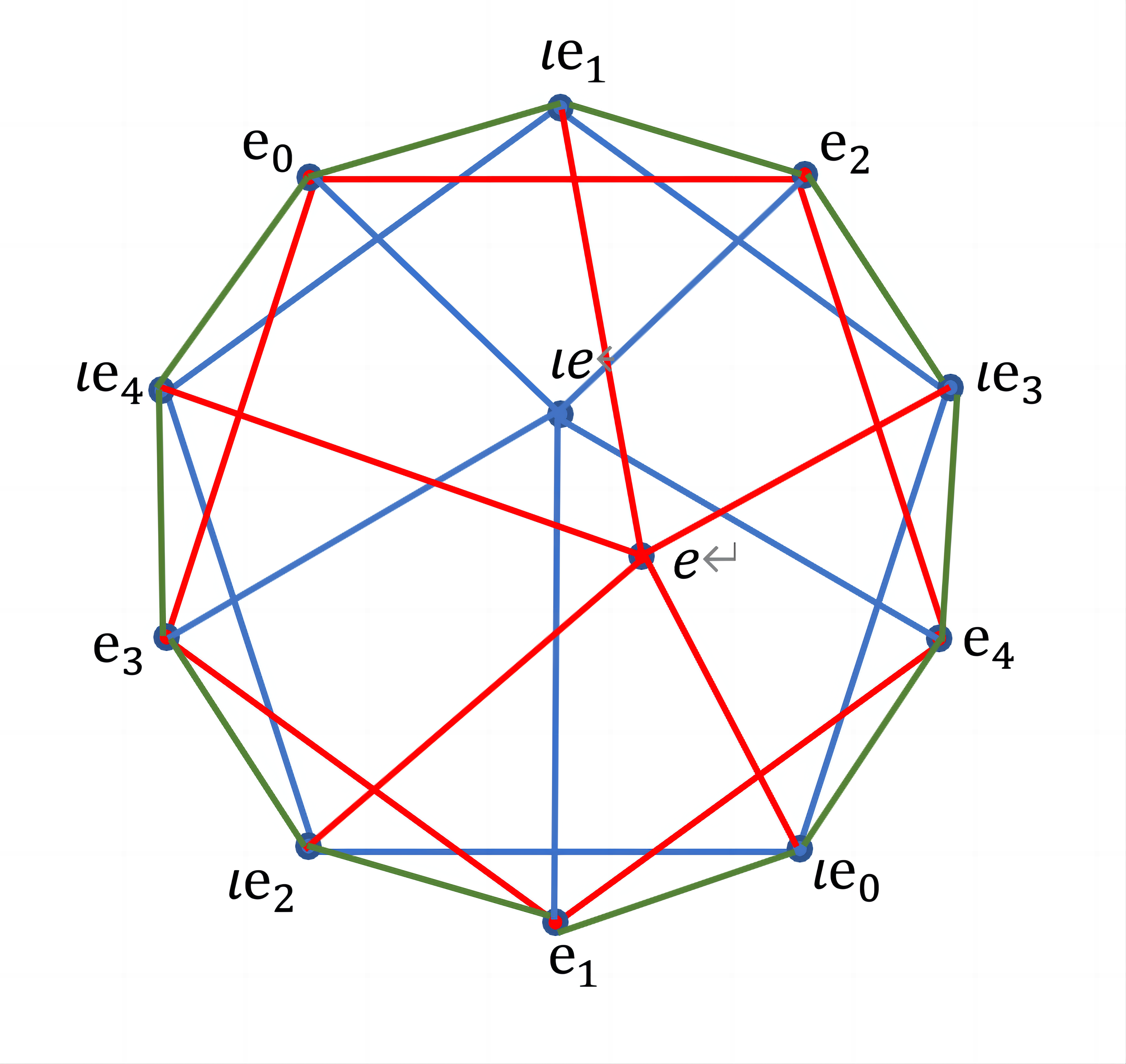}}
        \caption{\small{The intersection graph of 2-cells on $\tilde{\Pi}$. The vertices represents the 2-cells on $\tilde{\Pi}$ and two vertices are joined by an edge if the corresponding 2-cells intersect through an edge of $\tilde{\Pi}$.}}
    \end{figure}
    The proof for the second assertion is a direct construction. Observe that the intersection graph of 2-cells on $\tilde{\Pi}$ is as Figure \ref{intdode} where the vertices in the graph represent the 2-cells of $\tilde{\Pi}$ and two vertices are connected by an edge if and only if the 2-cells of $\tilde{\Pi}$ they represented intersect at a 1-cell of $\tilde{\Pi}$. Hence if we starting from the 2-cell $z_e$, we can reach $\iota z_e$ by crossing at least three 2-cells. From this observation, let $a\in \Rsr_0(\tilde{\Pi})$ be any fixed vertex which is not $e$. Let us choose a point $p_a$ lies on both $z_e$ and $\theta_a$. We also choose a small open band $\hat{U}_x$ of $\theta_{\iota x}$ in $\hat{\Pi}$ for each $x\in \Csr_0(\tilde{\Pi})$.
    We will let the path $\hat{\varepsilon}_a$ start from $p_a$ walk through the "shortest" path mentioned above while avoiding $\cup_{x\neq \pm a}U_x$ and ending in $\iota p_a$. Note that if $y$ is a 1-cell of edge type of $\hat{\Pi}$, $\hat{\varepsilon}_a$ intersects with $y$ only if $y$ has boundary point one on $z_e$ and another one on $z_{\iota e}$. The intersection point can be modified to lie in $U_{\iota a}$ and multiplicity is one.
    Hence from construction that the image $\varepsilon_a$ of $\hat{\varepsilon}_a$ in $\Pi$ will be a closed loop and the intersection numbers are as listed.
\end{proof}

\begin{proof}
    \textbf{(Proof of the Proposition \ref{Z1Pigen})}
    The proof for this Proposition is similar to the proof of Proposition \ref{OostrZ1}. We claim that the $\Zds$-module $\Hom_{\Zds\Isr}(E_o,Z_1(\Pi))$ is free generated by $\pi_{\bound}$, $\pi'_{\bound}$, $\pi_{\trc}$, $\pi'_{\trc}$, $\pi_{\edge}$ and $\pi$. This is showed by counting on the coefficients of 1-cells on $\Pi$.

    For the second part, we need to prove that the image of $\pi_{\edge}$, $\pi'_{\edge}$, $\pi_{\trc}$, $\pi$ generates the module $\Hom_{\Zds\Isr}(E_o,H_1(\Pi))$. This is done by checking the conditions in Lemma \ref{genhom}.
\end{proof}

\begin{cor}\label{Cor:PigenSL2O}
    Let $p$ is the natural map $Z_1(\Pi)\to H_1(\Pi)$ as above, the $\Isr$-equivariant morphism $U_{\Pi}$ and $V_{\Pi}$ in $\Hom_{\Zds\Isr}(E_o,H_1(\Pi))$ be defined as $U_{\Pi}:=p\circ \pi_{\trc}$ and $V_{\Pi}:=p\circ \pi$. The $\Ocal_o$-module $\Hom_{\Zds\Isr}(E_o,H_1(\Pi))$ is free $\Ocal_o$-module of rank two with generators $U_{\Pi}$ and $V_{\Pi}$.
\end{cor}
\begin{proof}
    We have proved in Proposition \ref{OostrZ1}, the $\Zds$-module $\Hom_{\Zds\Isr}(E_o,H_1(\Sigma))$ is freely generated by the image of $\pi_{\trc}$, $\pi'_{\trc}$, $\pi_{\edge}$ and $\pi$.
    Since the images of $\pi_{\bound}$ and $\pi'_{\bound}$ vanishes in $H_1(\Pi)$, we have $V_{\Pi}=\frac{1}{2}(p\circ\pi'_{\edge}+p\circ\pi_{\trc}-p\circ\pi'_{\trc})$. Combined with the Equations \eqref{Xpitrc} and \eqref{Xpiedge}, we have the following
    \begin{equation}\label{EQ:PigenSL2O}
        \begin{aligned}
             & p\circ\pi'_{\trc}     & = & XU_{\Pi}                   \\
             & p\circ\pi_{\edge}     & = & -(X-3)U_{\Pi}+(X-1)V_{\Pi} \\
             & p\circ\sigma'_{\edge} & = & (X-1)U_{\Pi}+2V_{\Pi}
        \end{aligned}
    \end{equation}
    Moreover it is clear to check that $U_{\Pi}$ and $V_{\Pi}$ are not linearly equivalent over $\Ocal_o$. These facts imply the Corollary.
\end{proof}

\begin{rmk}
    \textbf{(Other Models of the Bring's Curve)}
    In the article \cite{riera1992period}, G.\ Riera and R.\ Rodriguez introduced a hyperbolic model $\tilde{\Pi}^{\hyp}$ of the Bring's curve. This model also appears with great importance in \cite{braden2012bring}.
    It is a non-euclidean 20-gon lie on the Poincare's disk with the edges identified as in the Figure \ref{hypemod1}. It is known that the polygon's vertices fall into three equivalence classes $P_1$, $P_2$ and $P_3$ which is marked in the Figure \ref{hypemod1} and the genus of the curve is 4. The 20-gon can be tessellated by 240 triangles (or 120 double triangles) with interior angles $\frac{\pi}{5}$, $\frac{\pi}{4}$ and $\frac{\pi}{2}$ which is named as a \textit{(2,4,5)-triangle}. Hence it is clear to see that the tessellation of the 20-gon $\tilde{\Lambda}$ has 112 vertices which coming with three types.
    \begin{enumerate}
        \item The intersection of 4 (2,4,5)-triangles, the total number is 60,
        \item The intersection of 8 (2,4,5)-triangles, the total number is 30,
        \item The intersection of 10 (2,4,5)-triangles, the total number is 24.
    \end{enumerate}
    The $\Scal_5$-symmetry is given by permuting the double triangles.
    There are two kinds of regular hyperbolic pentagons on $\tilde{\Pi}^{\hyp}$. Twenty-four of them, which we call \textit{$\frac{\pi}{2}$-pentagon}, has inner angle $\frac{\pi}{2}$ which centered at the points of the third type and their vertices are always the points of the second type.
    Another twenty-four of them, which we call \textit{$\frac{2\pi}{5}$-pentagon}, has inner angle $\frac{2\pi}{5}$. Both centers and vertices are the points of the third type and the the midpoints of edges are of the second type.

    The hyperbolic model admits a Euclidean realization namely the great dodecahedron. The realization map is constructed by mapping 2-cells of $\tilde{\Pi}$ to the $\frac{2\pi}{5}$-pentagons in an $\Isr$-equivariant way.
    The points of the third type are divided into two disjoint 12-elements-sets, one is the images of vertices of $\tilde{\Pi}$ and another one is the images of barycenters of the faces. The $\iota$-map is induced from $\tilde{\Pi}$ in a natural way. We could remove in an $\Isr$-equivariant manner a small regular $\frac{\pi}{2}$-pentagon at each "vertices of $\tilde{\Pi}$". And identifying $\iota$ to get a model of genus 10 $\Isr$-curves.
    The advantage for this model is that we could see clearly the $\Scal_5$-symmetry on the Bring's curve. Note that the orientations of $\frac{\pi}{4}$ at the points of second type are only hyperbolic automorphisms and they are not euclidean.
    \begin{figure}
        \centering
        {\includegraphics[width=0.2\textwidth]{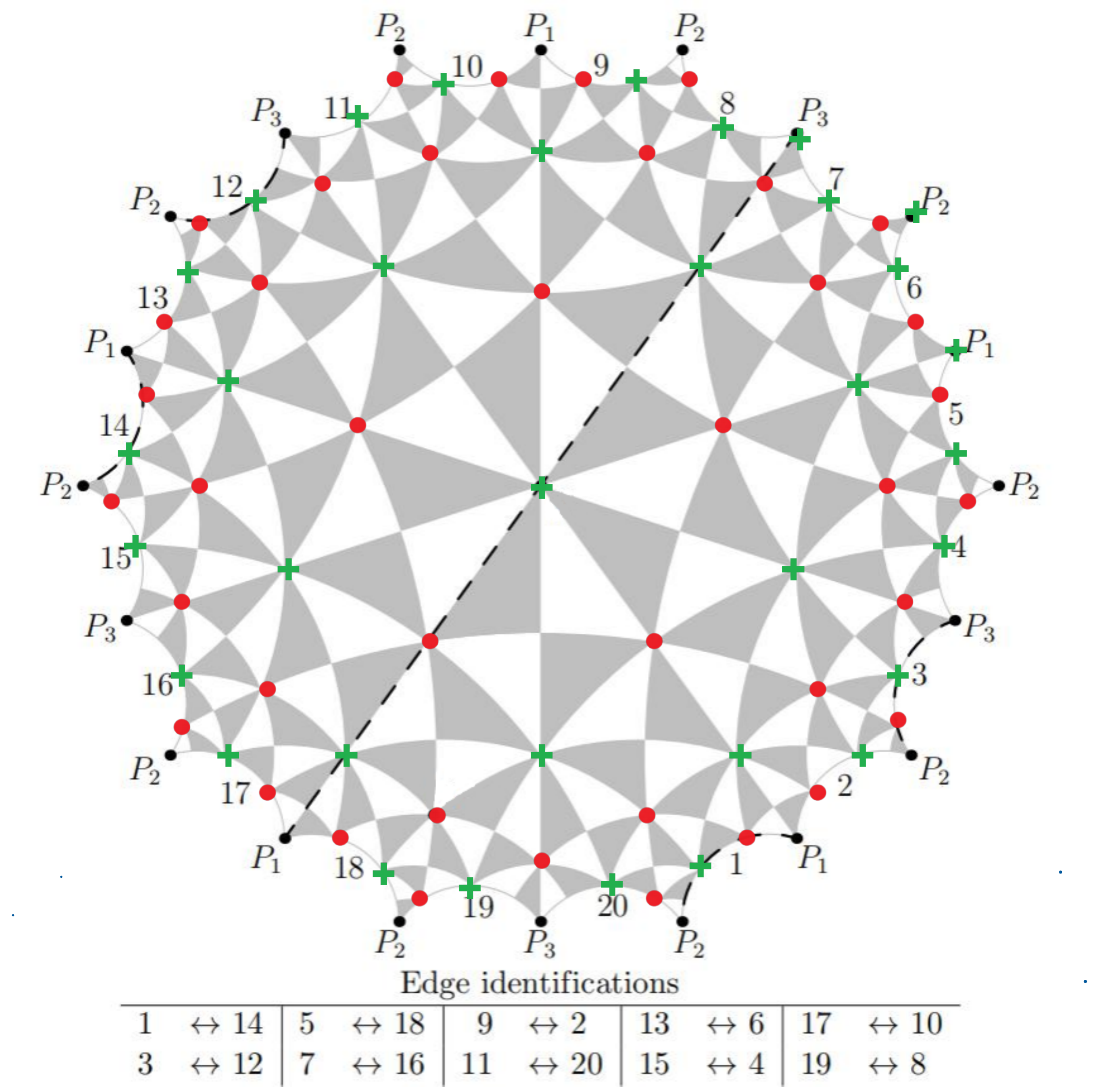}}
        \caption{\small{The edges of the 20-gon are identified as below. The points of second type are marked with \textcolor{red}{red point} and the points of third type are marked with \textcolor{green}{green plus}. The Figure is modified from the Figure 2 in \cite{braden2012bring}}}
        \label{Hyper-mod-Bri}
    \end{figure}
\end{rmk}


\section{Local Monodromy on the $E$-part}
Let us recall some basic ideas that we used in \cite{looijenga2021monodromy} and which is also useful in here.
Recall that on $\Sigma$ resp. $\Pi$ we defined a family of complex structures $J_\tau$ resp.$I_{\tau}$ with $\tau\in (0,1)$ resp. $\tau\in (0,1)$ which defined a path $\gamma_{\Sigma}: (0,1)\to \Bsr^\circ$ in the base of the Winger pencil  traversing the positive interval $(\infty, \frac{27}{5})$ resp. a path $\gamma_{\Pi}: (0,1)\to \Bsr^\circ$ in the base of the Winger pencil traversing the positive interval $(-1, \infty)$. This path had a continuous extension to $[0,1]$ resp.$[0,1]$ that gave rise to the stable degenerations $\Sigma_\edge$ (for $\gamma_{\Sigma}(0)=\infty$) and $\Sigma_\trc$ (for $\gamma_{\Sigma}(1)=\frac{27}{5})$ resp.$\Pi_\edge$ (for $\gamma_{\Pi}(0)=\infty$) and $\Pi_\trc$ (for $\gamma_{\Pi}(1)=-1)$.
We will determine the monodromies of these degenerations.
When determining the local monodromies given by $\Sigma$, it is convenient to regard $\gamma_{\Sigma}|(0,1)$ as a base point for $\Bsr^\circ$ and denote the fundamental group of $\Bsr^\circ$ with this base point by $\pi_{\Sigma}$.
Similarly when determining the local monodromies given by $\Pi$, we will regard $\gamma_{\Pi}|(0,1)$ as a base point for $\Bsr^\circ$. In this case we denote the fundamental group of $\Bsr^\circ$ with this base point by $\pi_{\Pi}$. They are parts of the monodromy representation of $\pi$ on $H_1(\Sigma)$. Clearly that $\pi_{\Sigma}$ and $\pi_{\Pi}$ are conjugate to each other. Hence We will denote this group by $\pi$ if there is no ambiguities.


Let $t\in\Bsr^{\circ}$, recall that we have an isotropic decomposition
\begin{equation}\label{candecom}
    H_1(C_t;\Qds)\cong(V_\Qds
    \otimes \Hom_{\Qds\Isr}(V_\Qds,H_1(C_t,\Qds)))\oplus(E_\Qds\otimes \Hom_{\Qds\Isr}(E_\Qds,H_1(C_t, \Qds)))
\end{equation}
Since the monodromy action will preserve this decomposition, we have a monodromy representation
of $\pi$ on both $\Hom_{\Qds\Isr}(V_\Qds,H_1(C_t;\Qds))$ and $\Hom_{\Qds\Isr}(E_\Qds,H_1(C_t; \Qds))$.
We have already determined the first type, together with an integral version of it $\Hom_{\Zds\Isr}(V_o,H_1(C_t))$ in \cite{looijenga2021monodromy}.
Here we will focus on the second type and its integral version i.e. $\Hom_{\Zds\Isr}(E_o,H_1(C_t))$. This integral global monodromy representation will be denoted by $\rho_{E_o}$. The space $\Hom_{\Qds\Isr}(E_\Qds,H_1(C_t,\Qds))$ is of dimension four over $\Qds$ since $E_{\Qds}$ admits non-trivial endomorphisms. However it will be of dimension two if treated as $\Kcal$-vector space, where $\Kcal$ is the endomorphism field of $E_\Qds$ defined in Corollary \ref{fieldKo}.
As we observed in Remark \ref{homsym} that the symplectic form on $H_1(C_t; \Qds)$ and the inner product on $E_\Qds$ give rise to a symplectic form on $\Hom_{\Qds\Isr}(E_\Qds, H_1(C_t; \Qds))$. The monodromies should keep this symplectic form, hence $\rho_{E_o}$ takes its values in $\Sp(1, \Ocal_o)\cong \SL_2(\Ocal_o)$.

Now let $C_s$ represents a singular member of the Winger pencil and $U_s\subset \Bsr$ a small disk-like neighborhood of $s$ (so that $C_s\subset \Wsr_{U_s}$ is a homotopy equivalence), we will determine $\rho_{E_o}$ locally for the degenerations $\Sigma_\trc$, $\Sigma_\edge$, $\Pi_{\trc}$, $\Pi_{\edge}$ and do a local discussion for degeneration of $3K$ in this section. If we choose $\alpha_s$ where $s\in\{\frac{27}{5},0,-1,\infty\}$ be a simple closed loop around $s$ only, the local fundamental groups is isomorphic to $\Zds$ with generator represented by $\alpha_s$ in these local cases. Hence the local monodromy around $s$ is determined by it value on $[\alpha_s]$.
The requirements of $U_s$ implies that for any $t\in U-\{s\}$ the natural map $H_1(C_t)\to H_1(\Wsr_{U})\cong H_1(C_s)$ is onto. So if $L$ denotes the kernel, then we get the short exact sequence
\begin{equation}\label{kHH}
    0\to L\to H_1(C_t)\to H_1(C_s)\to 0
\end{equation}
In case $C_s$ has only nodal singularities, $L$ is an $\Isr$-invariant isotropic primitive sublattice generated by the vanishing cycles. The monodromies will preserve this exact sequence and acts non-trivially only on the middle term. If we denote the set of vanishing cycles by $\Delta$ and $x$ be a class in $H_1(C_t)$, then the monodromies of $[\alpha_s]$ is given by the following well-known Picard-Lefschetz formula
\begin{equation}
    \rho_s(\alpha_s)(x)-x=\sum_{l\in\Delta/\{\pm 1\}}\la x,l\ra l
\end{equation}
These are the basic tools we will use in this section.

\subsection{The Monodromies of the Degenerations of $\Sigma$}\label{LmSig}
In this section, we will determine the local monodromies at the end points of $\gamma_{\Sigma}$. We have proved in Corollary \ref{Cor:SiggenSL2O} that $\Hom_{\Zds\Isr}(E_o,H_1(\Sigma))$ is free generated by $U_{\Sigma}$ and $V_{\Sigma}$ as $\Ocal_o$-module.
So it is natural to express the local monodromies in terms of these generators in this section. We will denote the local monodromies by $\rho_{\Sigma,\trc}$ and $\rho_{\Sigma,\edge}$ respectively and we will brief $p\circ\sigma_{\trc}$ to $\sigma_{\trc}$ which is the same for other symbols, since we will always work on $H_1(\Sigma)$ in this section.
The Theorem \ref{Thm:monoSig} below will give the local monodromy in each case.

\begin{thm}\label{Thm:monoSig}
    The monodromy $\rho_{\Sigma,\trc}$ fixes $U_{\Sigma}$ and brings $V_{\Sigma}$ to $(X-2)U_{\Sigma}+V$. The monodromy $\rho_{\Sigma,\edge}$ brings $U_{\Sigma}$ to $3U_{\Sigma}-(X+1)V_{\Sigma}$ and brings $V_{\Sigma}$ to $(X-1)U_{\Sigma}-V_{\Sigma}$.
\end{thm}

Let us recall some facts about the vanishing cycles $L_{\trc}$ resp. $L_{\edge}$ we discussed above before we give the proof of Theorem \ref{Thm:monoSig}.
Let $G_{\Sigma,\edge}$ be the dual intersection graph of $\Sigma_\edge$. It has six vertices and every two vertices are joined by an edge. Hence in this case we get the complete graph with six vertices, i.e. a graph of type $K_6$. If $\hat\Sigma_\edge$ is the normalization of the singular curve $\Sigma_{\edge}$, the set of connected components of $\hat\Sigma_\edge$ is denoted by $\Lsr$, then it has $6$ elements and $\Isr$ acts on it by permutations.
There is a natural homotopy class of maps $\Sigma_\edge\to G_{\Sigma,\edge}$ which induces an isomorphism $H_1(\Sigma_\edge)\to H_{1}(G_{\Sigma,\edge})$.
Recall that $H_1(G_{\Sigma,\edge})$ is free of rank 10, so that the kernel $L_{\Sigma,\edge}$ of $H_1(\Sigma)\to H_1(\Sigma_\edge)$ is in fact a primitive Lagrangian sublattice. The intersection product then identifies $L_{\Sigma,\edge}$ with the dual of $H_1(G_{\Sigma,\edge})$ so that the short exact sequence \eqref{kHH} becomes the following
\begin{equation}\label{0HHH0}
    \xymatrix{
    0\ar[r]& L_{\Sigma,\edge}\ar[r]& H_1(\Sigma)\ar[r]^-{\phi}& L_{\Sigma,\edge}^{\vee}\ar[r]& 0.
    }
\end{equation}
We have proved in \cite{looijenga2021monodromy} the following Lemma.
\begin{lem}\label{inftL}
    The natural homotopy class of maps $\Sigma_\edge\to G_{\Sigma,\edge}$ induces an isomorphism on $H^1$  and
    the  map which assigns to the ordered distinct pair $(l, l')$ in $\Lsr$ the $1$-cocycle on $G_{\Sigma,\edge}$ spanned by the vertices defined by $l$ and $l'$ induces an $\Isr$-equivariant isomorphism $\wedge^2W_0\cong H^1G_{\Sigma,\edge}$. If we call that $H^1(G_{\Sigma,\edge})$ is naturally identified with the vanishing homology of the degeneration $\Sigma$ into $\Sigma_\edge$, then this isomorphism identifies the set $\Delta_{\Sigma,\trc}$ of vanishing cycles with the set of unordered distinct pairs in $\Lsr$.
    Dually, $L_{\Sigma,\edge}^{\vee}=H_1(\Sigma_\edge)$ is as a $\Zds\Isr$-module isomorphic to $\wedge^2 W_o^{\vee}$.
\end{lem}

From this Lemma the short exact sequence \eqref{0HHH0} becomes the following sequence of $\Zds\Isr$-modules.
\begin{equation}\label{WHW}
    \xymatrix{
    0\ar[r]& \wedge^2 W_o\ar[r]\ar[dr]_-{j_{\Sigma,\edge}}& H_1(\Sigma)\ar[r]& (\wedge^2 W_o)^\vee\ar[r]& 0\\
    &&Z_1(\Sigma)/B_1(\Sigma)\ar[u]&&
    }
\end{equation}
Note that $\wedge^2 W_o$ has a single generator as a  $\Zds\Isr$-module, for example $\bar{l} \wedge \bar{l'}$ with $l,l'$ distinct. We have an $\Isr$-isomorphism $\iota_\edge:\wedge^2 W_o\to Z_1(\Sigma)$  which sends $\bar{l}\wedge \bar{l'}$ to the element in $Z_{\edge}(\Sigma)$ with the same stabilizer.
Let us apply the left exact functor $\Hom_{\Zds\Isr}(E_o,\cdot)$ to the short exact sequence \eqref{WHW} and combine it with the exact sequence \eqref{seqph1}
\begin{equation}\label{hominf}
    \xymatrix@C=1pc@R=2pc{
    0\ar[r]& \Hom_{\Zds\Isr}(E_o,\wedge^2 W_o)\ar[r]\ar[dr]_-{j_{\Sigma,\edge,*}}&\Hom_{\Zds\Isr}(E_o,H_1(\Sigma))\ar[r]^-{\phi}&\Hom_{\Zds\Isr}(E_o, \wedge^2 W_o^{\vee})\ar[r]& \Ext_{\Zds\Isr}(E_o,\wedge^2 W_o)\\
    &&\coker(i)\ar[u]&&
    }
\end{equation}
By Proposition \ref{OostrZ1}, the vertical arrow
\begin{equation*}
    \coker(i)=\Hom_{\Zds\Isr}(E_o,Z_1(\Sigma))/\Hom_{\Zds\Isr}(E_o,B_1(\Sigma))\to \Hom_{\Zds\Isr}(V_o,H_1(\Sigma))
\end{equation*}
is an isomorphism.

Likewise at the other end: if $G_{\Sigma,\trc} $ is the dual intersection graph of $\Sigma_\trc$, then the kernel of $H_1(\Sigma)\to H_1(\Sigma_\trc)\cong H_1({\Sigma,\trc})$ is the primitive Lagrangian sublattice $L_{\Sigma,\trc}$ we introduced earlier and we get a similar short exact sequence and a similar description of the associated monodromy $\rho_{\Sigma,\trc}$ in terms of $\Delta_{\Sigma,\trc}$.

The short exact sequence \eqref{kHH} becomes the following short exact sequence of $\Zds\Isr$-modules
\begin{equation}\label{LHLv}
    \xymatrix{
    0\ar[r]& L_{\Sigma,\trc}\ar[r]\ar[dr]_{j_{\Sigma,\trc}}& {H_1(\Sigma,\Zds)}\ar[r]^-{\phi}& L_{\Sigma,\trc}^{\vee}\ar[r]& 0\\
    &&{Z_1(\Sigma)/B_1(\Sigma)}\ar[u]&&
    }
\end{equation}
Here $j_{\Sigma,\trc}$ is the obvious map. Applying the left exact functor $\Hom_{\Zds\Isr}(E_o,\cdot)$ to the short exact sequence \eqref{LHLv} and combine it with the exact sequence \eqref{seqph1}
\begin{equation}\label{hom275}
    \xymatrix@C=1pc@R=2pc{
    0\ar[r]& \Hom_{\Zds\Isr}(E_o,L_{\Sigma,\trc})\ar[r]\ar[dr]_{j_{\Sigma,\trc\ *}}&\Hom_{\Zds\Isr}(E_o,H_1(\Sigma))\ar[r]^-{\phi_*}&\Hom_{\Zds\Isr}(E_o, L_{\Sigma,\trc}^{\vee})\ar[r]& \Ext_{\Zds\Isr}(E_o,L_{\Sigma,\trc})\\
    &&\coker(i)\ar[u]&&
    }
\end{equation}
The vertical arrow is an isomorphism same as above.

\begin{proof}
    \textbf{(Proof of Theorem \ref{Thm:monoSig})}

    By the Proposition \ref{degsig} the images of $\sigma_{\edge}$ and $\sigma'_{\edge}$ lie in $L_{\Sigma,\edge}$ and the image of $\sigma_{\trc}$ and $\sigma'_{\trc}$ lie in $L_{\Sigma,\trc}$. Hence the monodromy $\rho_{\Sigma,\trc}$ fixes  $U_{\Sigma}=\sigma_{\trc}$ and $\sigma'_{\trc}$, while $\rho_{\Sigma,\edge}$ fixes $\sigma_{\edge}$ and $\sigma'_{\edge}$.

    By the Picard-Lefschetz formula
    \begin{equation*}
        \rho_{\Sigma,\trc}(V_{\Sigma}(e))- (V_{\Sigma}(e))                  =  \sum_{\delta\in\Delta_{\Sigma,\trc}/\{\pm1\}}\la [V_{\Sigma}(e)],\delta\ra\delta          =  \frac{1}{2}\sum_{x\in\Rsr_0(\tilde{\Sigma})}\la [\sigma'_{\edge}(e)],\delta_x\ra\delta_x
    \end{equation*}
    Let $x\in \Rsr_2(\tilde{\Sigma})\subset\Csr_2(\tilde{\Sigma})$, from the Lemma \ref{speint} $\la [\sigma'_{\edge}(e)],\delta_x\ra$ equals 2 if $x\notin e$, otherwise it is $0$. Hence $\sum_{x\in\Rsr_0(\Sigma)}\la [\sigma'_{\edge}(e)],\delta_x\ra\delta_x$ equals $\delta'_{\trc}$. Hence we have $\rho_{\Sigma,\trc}(V_{\Sigma})=V_{\Sigma}+\sigma'_{\trc}=(X-2)U_{\Sigma}+V_{\Sigma}$ from Equation \eqref{EQ:SiggenSL2O}.

    Similarly, for $\rho_{\Sigma,\edge}$ we have
    \begin{equation*}
        \begin{aligned}
            \rho_{\Sigma,\edge}(U_{\Sigma}(e))- U_{\Sigma}(e) & = & \sum_{\delta\in\Delta_{\Sigma,\edge}/\{\pm1\}}\la [U_{\Sigma}(e)],\delta\ra\delta & = & \sum_{y\in\Rsr_1(\tilde{\Sigma})}\la [\sigma_{\trc}(e)],\delta_y\ra\delta_y                                \\
            \rho_{\Sigma,\edge}(V_{\Sigma}(e))- V_{\Sigma}(e) & = & \sum_{\delta\in\Delta_{\Sigma,\edge}/\{\pm1\}}\la [V_{\Sigma}(e)],\delta\ra\delta & = & \frac{1}{2}\sum_{y\in\Rsr_1(\tilde{\Sigma})}\la [\sigma_{\trc}(e)]+[\sigma'_{\trc}(e)],\delta_y\ra\delta_y \\
        \end{aligned}
    \end{equation*}
    From the Lemma \ref{speint}, $\la [\sigma_{\trc}(e)],\delta_y\ra$ equals $-1$ if $\ini(y)\in e$ but $\tm(y)\notin e$ otherwise it equals 0. And $\la [\sigma'_{\trc}(e)],\delta_y\ra$ equals $1$ if $\ini(y)\in e$ but $\tm(y)\notin e$, $-2$ if $\ini(y)\notin e$ and $0$ in other cases. Hence we have $\sum_{y\in\Rsr_1(\tilde{\Sigma})}\la [\sigma_{\trc}(e)],\delta_y\ra\delta_y$ is $-\delta_{\edge}$, $\sum_{y\in\Rsr_1(\tilde{\Sigma})}\la [\sigma'_{\trc}(e)],\delta_y\ra\delta_y$ is $\delta_{\edge}-2\delta'_{\edge}$. Therefore we have $\rho_{\Sigma,\edge}(U_{\Sigma})=U_{\Sigma}-\sigma_{\edge}=3U_{\Sigma}-(X+1)V_{\Sigma}$ and $\rho_{\Sigma,\edge}(V_{\Sigma})=V_{\Sigma}-\sigma'_{\edge}=(X-1)U_{\Sigma}-V_{\Sigma}$. This finishes the proof.
\end{proof}

\subsection{The Monodromies of the Degenerations of $\Pi$}

In this section, we will determine the local monodromies at the end points of $\gamma_{\Pi}$.
The Theorem \ref{Thm:monoPi} below will give the local monodromy in each case. Recall that by Corollary \ref{Cor:PigenSL2O}, the module $\Hom_{\Zds\Isr}(E_o,H_1(\Pi))$ is freely generated by $U_{\Pi}$ and $V_{\Pi}$ as $\Ocal_o$-module. So we will express the monodromies $ \rho_{\Pi,\edge}$ and $\rho_{\Pi,\trc}$ in terms of these generators when computing the local monodromies defined by $\Pi$.

\begin{thm}\label{Thm:monoPi}
    The monodromy $\rho_{\Pi,\trc}$ fixes $U_{\Pi}$ and takes $V_{\Pi}$ to $XU_{\Pi}+V_{\Pi}$.
    The monodromy $\rho_{\Pi,\edge}$ brings $U_{\Pi}$ to $(X-2)U_{\Pi}-(X-1)V_{\Pi}$ and $V_{\Pi}$ to $(2X-4)U_{\Pi}+(4-X)V_{\Pi}$.
\end{thm}

It is clear that the dual intersection graph $G_{\Pi,\edge}$ of $\Pi_\edge$ is the same as $G_{\Sigma,\edge}$.Therefore we have the same results as the exact sequences \eqref{0HHH0}, \eqref{WHW} and \eqref{seqph1} with $\Sigma$ replaced by $\Pi$.
There is some difference at the other end: if $G_{\Pi,\trc} $ is the dual intersection graph of $\Pi_\trc$, $G_{\Pi,\trc} $ has only one vertex and six edges with the vertex marked with 4. In this case the kernel $L_{\Pi,\trc}$ of $H_1(\Pi)\to H_1(\Pi_\trc)$ in the exact sequence \eqref{kHH} is generated by 6 elements i,e the vanishing cycles which denote their collection by $\Delta_{\Pi,\trc}$.
Then it is a primitive isotropic sublattice $L_{\Pi,\trc}$ of rank six which is not Lagrangian. Hence the exact sequence \eqref{kHH} will become the following.
\begin{equation}\label{0HHH0Pitrc}
    \xymatrix{
    0\ar[r]& L_{\Pi,\trc}\ar[r]& H_1(\Pi)\ar[r]^-{\phi}& H_1(C_{-1})\ar[r]& 0.
    }
\end{equation}
However we have a similar description of the associated monodromy $\rho_{\Pi,\trc}$ in terms of $\Delta_{\Pi,\trc}$.
\begin{equation}\label{LHLvPi}
    \xymatrix{
    0\ar[r]& L_{\Pi,\trc}\ar[r]\ar[dr]_{j_{\Pi,\trc}}& {H_1(\Pi,\Zds)}\ar[r]^-{\phi}& H_1(C_{-1})\ar[r]& 0\\
    &&{Z_1(\Pi)/B_1(\Pi)}\ar[u]&&
    }
\end{equation}
Here $j_{\Pi,\trc}$ is the obvious map. Applying the left exact functor $\Hom_{\Zds\Isr}(E_o,\cdot)$ to the short exact sequence \eqref{LHLvPi} and combine it with the exact sequence \eqref{eqn:bacicexseqPi}
\begin{equation}\label{hom-1}
    \xymatrix@C=1pc@R=2pc{
    0\ar[r]& \Hom_{\Zds\Isr}(E_o,L_{\Pi,\trc})\ar[r]\ar[dr]_{j_{\Pi,\trc\ *}}&\Hom_{\Zds\Isr}(E_o,H_1(\Pi))\ar[r]^-{\phi_*}&\Hom_{\Zds\Isr}(E_o, H_1(C_{-1}))\ar[r]& \Ext_{\Zds\Isr}(E_o,L_{\Sigma,\trc})\\
    &&\coker(i)\ar[u]&&
    }
\end{equation}
The vertical arrow is an isomorphism same as above.

\begin{proof}
    \textbf{(Proof of Theorem \ref{Thm:monoPi})} The proof is similar to the proof of Theorem \ref{Thm:monoSig}. By the Proposition \ref{hypeModV} the image of $\pi_{\edge}$ and $\pi'_{\edge}$ lie in $L_{\Pi,\edge}$ and the image of $\pi_{\trc}$ and $\pi'_{\trc}$ lie in $L_{\Pi,\trc}$. Hence the monodromy $\rho_{\Pi,\edge}$ fixes $\pi_{\edge}$ and $\pi'_{\edge}$, while the monodromy $\rho_{\Pi,\trc}$ fixes  $U_{\Pi}=\pi_{\trc}$ and $\pi'_{\trc}$.

    Recall that we take $\Rsr_i(\tilde{\Pi})$ be systems of representatives of $\iota$-symmetry on $\Csr_i(\tilde{\Pi})$. By the Picard-Lefschetz formula
    \begin{equation*}
        \rho_{\Pi,\trc}(V_{\Pi}(e))- V_{\Pi}(e)                 =  \sum_{\theta\in\Delta_{\Pi,\trc}/\{\pm1\}}\la [V_{\Pi}],\theta\ra\theta          =  \frac{1}{2}\sum_{x\in\Rsr_0(\tilde{\Pi})}\la [\pi'_{\edge}(e)],\theta_x\ra\theta_x
    \end{equation*}
    Let $x\in \Rsr_0(\tilde{\Pi})\subset\Csr_0(\tilde{\Pi})$, from the Lemma \ref{dulintPi} $\la [\pi'_{\edge}(e)],\theta_x\ra$ equals $2$ if $x\neq e$, otherwise it is $0$. Hence the sum $\sum_{x\in\Rsr_0(\tilde{\Pi})}\la [\pi'_{\edge}(e)],\theta_x\ra\theta_x$ equals $2\theta'_{\trc}$. Therefore we have $\rho_{\Pi,\trc}(V_{\Pi}(e))=V_{\Pi}+\pi'_{\trc}=XU_{\Pi}+V_{\Pi}$ from Equation \eqref{EQ:PigenSL2O}.

    Similarly, for $\rho_{\Pi,\edge}$ we have
    \begin{equation*}
        \begin{aligned}
            \rho_{\Pi,\edge}(U_{\Pi}(e))- U_{\Pi}(e) & = & \sum_{\theta\in\Delta_{\Pi,\edge}/\{\pm1\}}\la [U_{\Pi}(e)],\theta\ra\theta & = & \sum_{y\in\Rsr_1(\tilde{\Pi})}\la [\pi_{\trc}(e)],\theta_y\ra\theta_y                             \\
            \rho_{\Pi,\edge}(V_{\Pi}(e))- V_{\Pi}(e) & = & \sum_{\theta\in\Delta_{\Pi,\edge}/\{\pm1\}}\la [V_{\Pi}(e)],\theta\ra\theta & = & \frac{1}{2}\sum_{y\in\Rsr_1(\tilde{\Pi})}\la [\pi_{\trc}(e)]-[\pi'_{\trc}(e)],\theta_y\ra\theta_y \\
        \end{aligned}
    \end{equation*}
    From the Lemma \ref{dulintPi}, $\la [\pi_{\trc}(e)],\theta_y\ra$ equals $-1$ if the initial point of $y$ is $e$ otherwise it equals 0. And $\la [\pi'_{\trc}(e)],\theta_y\ra$ equals $1$ if the initial point of $y$ is $e$, $2$ if the initial point of $y$ lies on $P_e$ while the terminal point of $y$ lies on $P_{\iota e}$ and $0$ in other cases. Hence we have $\sum_{y\in\Rsr_1(\tilde{\Pi})}\la [\pi_{\trc}(e)],\theta_y\ra\theta_y$ equals to $-\theta_{\edge}$ and $\sum_{y\in\Rsr_1(\tilde{\Pi})}\la [\pi'_{\trc}(e)],\theta_y\ra\theta_y$ equals to $\theta_{\edge}-2\theta'_{\edge}$. Therefore we have the monodromies $\rho_{\Pi,\edge}(U_{\Pi})=U_{\Pi}-\pi_{\edge}=(X-2)U_{\Pi}-(X-1)V_{\Pi}$ and $\rho_{\Pi,\edge}(V_{\Pi})=V_{\Pi}-\pi_{\edge}+\pi'_{\edge}=(2X-4)U_{\Pi}+(4-X)V_{\Pi}$. This finishes the proof.
\end{proof}

\subsection{The Local Monodromies Near the Triple Conic}
We claim that the monodromy around $s = 0$ is of order three. Remember that $C_0$ is the unstable curve $3K$, where $K$ is an $\Isr$-invariant (smooth) conic.
Let ${U_0} \subset\Bsr$ be an open disk centered at $s = 0$ of radius $< \frac{27}{5}$. We proved in \cite{zi2021geometry} that by doing a base change over ${U_0}$ of order 3 (with Galois group $\mu_3$), given by $\hat{t} \in \hat{U}_0 \to t = \hat{t}^3\in {U_0}$, the pull back of $\Wsr_{U_0} /{U_0}$ can be modified over the central fiber $C_0$ only to make it a smooth family $\hat{\Wsr}_{\hat{U}_0}/\hat{U}_0$ which still retains the $\mu_3$-action. The central fiber is then a smooth curve $\hat{C}_0$ with an action of $\Isr \times \mu_3$ whose $\mu_3$-orbit space gives $K$. This implies that the monodromy of the original family around $0$ (which is a priori only given as an isotopy class of diffeomorphisms of a nearby smooth fiber) can be represented by the action of a generator $\phi\in \mu_3$ on $\hat{C}_0$ (which indeed commutes with the $\Isr$-action on $C_0$).

\begin{cor}\label{cor:rhocyc3}
    Let $t\in {U_0}\backslash\{0\}\subset\Bsr$, the monodromy automorphism $\rho_0$ acts on $\Hom_{\Zds\Isr}(E_o,H_1(C_t))$ with order three.
\end{cor}
\begin{proof}
    This comes from the Corollary 4.8 of \cite{looijenga2021monodromy}.
\end{proof}


\section{Global Monodromy and Period Map on the $E$-part}

We will determine the global monodromy group and the period map in this section. Let us take $\psi$ and $\psi'$ be the two naturally defined embeddings of $\Kcal\hookrightarrow \Rds$. It is clear that $\psi$ and $\psi'$ induce two different embeddings $\SL_2(\Ocal_o)\hookrightarrow \SL_2(\Rds)$ which we still denote they by $\psi$ and $\psi'$. Hence the map $(\psi,\psi')$ will embed the group $\SL_2(\Ocal_o)$ into $\SL_2(\Rds)$ with the diagonal isomorphic to $\SL_2(\Zds)$. This could also be described as follows: there exist a Galois involution $\varphi$ of $\SL_2(\Ocal_o)$ which will exchange the image of the two embeddings in $\SL_2(\Rds)\times \SL_2(\Rds)$, the fixed points are the group $\SL_2(\Zds)$.
Moreover we could observe that $\SL_2(\Ocal_o)$ acts faithfully and discontinuously on $\Hds^2$ through this embedding. The quotient $\SL_2(\Ocal_o)/\Hds^2$ is then a algebraic surface called \textit{Hilbert's modular surface}.
We will need the Theorem 4.6 in \cite{farb2021Arithmeticity} listed below which is also a special case of the main theorem of \cite{benoist2010discreteness}.
\begin{thm}\label{Thm:cerfi}
    Let $K$ be a real quadratic number field, $\Ocal_K$ its ring of integers and $\Omega<K$ a lattice. Let $\Lambda< \SL_2(\Ocal_K)$ be the subgroup generated by matrix of the form $\begin{pmatrix}
            a & b \\
            c & d
        \end{pmatrix}$ with $c\neq 0$, together with the set of matrices
    \begin{equation*}
        \{\begin{pmatrix}
            1,\omega \\
            0,1
        \end{pmatrix}:\omega\in\Omega\}
    \end{equation*}
    If $\psi,\ \psi':K\to\Rds$ are the two real embeddings of $K$, then the associated embedding $\SL_2(\Ocal_K)\to \SL_2(\Rds)\times\SL_2(\Rds)$ maps $\Lambda$ onto a lattice in $\SL_2(\Rds)\times\SL_2(\Rds)$. In particular $\Lambda$ has finite index in $\SL_2(\Ocal_K)$.
\end{thm}
Since the two models $\Sigma$ and $\Pi$ gives the same singular fiber at the "$\edge$" ends, it is clear that $\rho_{\Sigma,\edge}$ and $\rho_{\Pi,\edge}$ should conjugate to each other by a transformation in $\Sp_1(\Ocal_o)\cong \SL_2(\Ocal_o)$. It is clear to check that if we take the linear transformation $P$ as Equations \eqref{Eq:P}, we will have $\rho_{\Pi,\edge}=P^{-1}\rho_{\Sigma,\edge}P$.
\begin{equation}\label{Eq:P}
    \begin{aligned}
        PU_{\Pi} & = & -V_{\Sigma}           \\
        PV_{\Pi} & = & U_{\Sigma}-V_{\Sigma}
    \end{aligned}
\end{equation}
From this observation, we will take the basis $(U:=U_{\Sigma},V:=V_{\Sigma})$ as a basis for $\Hom_{\Zds\Isr}(E_o,H_1(C_t))$.
First we give a direct computation to the Corollary \ref{cor:rhocyc3} that we proved in the last section.
\begin{cor}
    The monodromy $\rho_0$ action on $\Hom_{\Zds\Isr}(E_o,H_1(C_t))$ is cyclic of order three. More explicitly, it brings $U$ to $-U+V$ and $V$ to $-U$.
\end{cor}
\begin{proof}
    Since the smooth locus $\Bsr^{\circ}$ is obtained by removing four points from $\Pds^1$ we will have the following equation
    \begin{equation*}
        \rho_0^{-1}=\rho_{\Sigma,\trc}\rho_{\Sigma,\edge}P^{-1}\rho_{\Pi,\trc}P
    \end{equation*}
    Then the computation shows that $\rho_0^{-1}$ is given as following:
    \begin{equation*}
        \begin{aligned}
            \rho_0^{-1}(U) & = & -V  \\
            \rho_0^{-1}(V) & = & U-V
        \end{aligned}
    \end{equation*}
    Hence it is cyclic of order three.
\end{proof}
\begin{thm}
    The monodromy group $\Gamma_{E_o}$ is a subgroup of finite index in $\SL_{2}(\Ocal_o)$. In particular it is arithmetic.
\end{thm}
\begin{proof}
    In order to show this theorem we only need to check the conditions in Theorem \ref{Thm:cerfi}. From the above observation that $\Gamma_{E_o}$ is generated by the following three generators
    \begin{equation}
        \rho_{\Sigma,\trc}=\begin{pmatrix}
            1 & X-2 \\
            0 & 1
        \end{pmatrix},\ \rho_{\Sigma,\edge}=\begin{pmatrix}3&X-1\\-(1+X)&-1\end{pmatrix},\ P^{-1}\rho_{\Pi,\trc}P=\begin{pmatrix}1&0\\-X&1\end{pmatrix}
    \end{equation}
    It is clear that $\Gamma_E$ is generated by matrix of the form
    $\begin{pmatrix}
            a & b \\
            c & d
        \end{pmatrix}$
    with $c\neq 0$ and upper triangular matrices. Hence $\Gamma_{E_o}$ is a finite index subgroup of $\SL_2(\Ocal_o)$.
\end{proof}

Finally summarize all the facts about the monodromy, we could determine the 'partial' period map. Let $\Bsr^{+}$ be the open subvariety of $\Bsr$ obtained by removing from $\Bsr$ the three points representing nodal curves.
\begin{thm}
    The 'partial' period map $p_{E_o}:\Bsr^{+}\to\Gamma_{E_o}/\Hds^2 \to\SL_2(\Ocal_o)/\Hds^2$ has the property that the first arrow is open and the second map is finite.
\end{thm}

\section{Computation for the Index of $\Gamma_{E_o}$ in $\SL_2(\Ocal_o)$}

Recall that $\Ocal_o=\Zds[X]/(X^2-5)$ is the endomorphisms ring of $\Zds\Acal_5$-module $E_o$ and $\Ocal=\Zds[Y]/(Y^2-Y-1)$ is isomorphic to the ring of integers in the algebraic field $\Qds[\sqrt{5}]$. Their relations are as following:
the natural map given by quotient $2\Ocal$ gives the exact sequence
\begin{equation*}
    \Ocal\to\Ocal/2\Ocal\to 1
\end{equation*}
It has the properties that the last term $\Ocal/2\Ocal$ is isomorphic to $\Fds_4$ and $\Ocal_o$ is the pullback of $\Fds_2\subset\Fds_4$.
The similar properties also holds if we consider special linear groups with entries in $\Ocal$ and $\Ocal_o$. We have the following exact sequence of the groups
\begin{equation}\label{SLO}
    \SL_2(\Ocal)\to\SL_2(\Ocal/2\Ocal)\cong \SL_2(\Fds_4)\to 1
\end{equation}
The subgroup $\SL_2(\Ocal_o)$ is the pullback of $\SL_2(\Fds_2)\subset \SL_2(\Fds_4)$.
From these facts, we have the following proposition
\begin{prop}\label{indexSLOoSLO}
    The Exact Sequence \ref{SLO} induced an one-to-one correspondence of sets of left cosets
    \begin{equation*}
        \SL_2(\Ocal)/\SL_2(\Ocal_o)\to \SL_2(\Fds_4)/\SL_2(\Fds_2)
    \end{equation*}
    In particular $\SL_2(\Ocal_o)$ has index 10 in $\SL_2(\Ocal)$.
\end{prop}
\begin{proof}
    It is clear to check that this map is well-defined, surjective and injective. The index comes from the facts that $\SL_2(\Fds_4)$ is isomorphic to $\Acal_5$ and $\SL_2(\Fds_2)$ is isomorphic to $\Scal_3$.
\end{proof}
Besides we have an explicit description for the generating set of $\SL_2(\Ocal)$ which we will consider the following matrices in $\SL_2(\Ocal)$:
\begin{equation*}
    \begin{aligned}
         & A_0  =  -\Id,\                         &                                      \\
         & A_1  =  \begin{pmatrix}
            0  & 1 \\
            -1 & 0
        \end{pmatrix},\  & A_2  =   \begin{pmatrix}
            1 & 1 \\
            0 & 1
        \end{pmatrix} \\
         & A_3  =  \begin{pmatrix}
            X & 0   \\
            0 & X-1
        \end{pmatrix},\  & A_4  =   \begin{pmatrix}
            1 & X \\
            0 & 1
        \end{pmatrix} \\
    \end{aligned}
\end{equation*}
The following Proposition is the Corollary 2.3 in \cite{stover2021geometry} which showed that $\SL_2(\Ocal)$ is generated by all the $A_i$s together with $-\Id$ subject to some relations.
\begin{prop}\label{genSLO}
    The group $\SL_2(\Ocal)$ is generated by $A_0,\cdots, A_4$ subject to the following relations
    \begin{table}[h]
        \centering
        \begin{tabular}{lll}
            \noalign{\smallskip}
            $C_0  =A_0^2$,                                        & \  &                                      \\
            \noalign{\smallskip}
            $C_1  =[A_0,A_1]$,                                    & \  & $C_2=[A_0,A_2]$,                     \\
            \noalign{\smallskip}
            $C_3  =[A_0,A_3]$,                                    & \  & $C_4=[A_0,A_4]$,                     \\
            \noalign{\smallskip}
            $R_1  =A_0A_1^2$,                                     & \  & $R_2=(A_1A_2)^3$,                    \\
            \noalign{\smallskip}
            $R_3  =A_0(A_1A_3)^2$,                                & \  & $R_4=[A_2,A_4]$,                     \\
            \noalign{\smallskip}
            $R_5  =A_3A_2A_3^{-1}(A_2A_4)^{-1}$,                  & \  & $R_6=A_3A_4A_3^{-1}(A_2A_4^2)^{-1}$, \\
            \noalign{\smallskip}
            $R_7  =A_0A_1A_4A_1(A_2A_4^{-1}A_1A_4^{-1}A_3)^{-1}$, & \  & \
        \end{tabular}
    \end{table}

\end{prop}
\begin{thm}
    The monodromy group $\Gamma_{E_o}$ is a subgroup of index 20 in $\SL_{2}(\Ocal)$. Hence it has index 2 in $\SL_2(\Ocal_o)$.
\end{thm}
\begin{proof}
    We first consider the index of $\Gamma_{E_o}$ in $\SL_2(\Ocal)$. It is clear to check that  $\rho_{\Sigma,\trc}=A_4^2 A_2^{-3}$, $\rho_{\Sigma,\edge}= A_4^{-2} A_2^2 A_1 A_4^{-2} A_1$ and $P^{-1}\rho_{\Pi,\trc}P= -A_1 A_4^2 A_2^{-1} A_1$. Using the \textit{Index} function in the computer program \textit{Magma} \cite{bosma1997magma}, we could compute that the index $[\SL_2(\Ocal):\Gamma_{E_o}]$ is 20. The Program has been uploaded to \cite{indexprog}.
    From the Proposition \ref{indexSLOoSLO} the index $[\SL_2(\Ocal):\SL_2(\Ocal_o)]$ is 10, hence the index $[\SL_2(\Ocal_o):\Gamma_{E_o}]$ is 2.
\end{proof}

\bibliographystyle{plain}
\bibliography{./bib/Reference}

\end{document}